\documentclass[a4paper]{amsart}

\usepackage[all]{xy}
\usepackage{pgf,tikz}
\usepackage{tikz-cd}
\usepackage{pgfplots}

\newcommand*{\Hc}{\mathcal{H}}
\newcommand*{\K}{\mathcal{K}}

\newcommand*{\T}{\mathcal{T}}
\newcommand*{\Bf}{{\bf f}}
\newcommand*{\Bg}{{\bf g}}
\newcommand*{\BC}{{\bf C}}
\newcommand*{\BD}{{\bf D}}

\newcommand*{\ket}{\rangle}
\newcommand*{\bra}{\langle}
\newcommand{\GG}{\mathbb{G}}
\newcommand{\HH}{\mathbb{H}}

\DeclareMathOperator{\Aut}{Aut}

\DeclareMathOperator{\Hom}{Hom}
\DeclareMathOperator{\Mor}{Mor}
\DeclareMathOperator{\id}{id}
\DeclareMathOperator{\tr}{tr}
\DeclareMathOperator{\ord}{ord}
\DeclareMathOperator{\Hilb}{Hilb}
\DeclareMathOperator{\HILB}{HILB}
\DeclareMathOperator{\Perm}{Perm}
\DeclareMathOperator{\PERM}{PERM}
\DeclareMathOperator{\Homeo}{Homeo}
\DeclareMathOperator{\HOMEO}{HOMEO}
\DeclareMathOperator{\Deck}{Deck}
\DeclareMathOperator{\DECK}{DECK}

\DeclareMathOperator{\Sym}{Sym}
\DeclareMathOperator{\Qut}{Qut}

\DeclareMathOperator{\Rep}{Rep}

\newenvironment{bnum}
{\begin{list}{}
    {\setlength{\labelwidth}{15pt}
     \setlength{\leftmargin}{\labelwidth}
    }
}
{\end{list}}

\newcommand{\tp}{\!\!
{\scriptstyle
\text{
\raisebox{0.7pt}{
\textcircled{\raisebox{-1.3pt}{$\mathsf{T}$}}
} 
} 
} 
\!\!}

\numberwithin{equation}{section}
\theoremstyle{change}
\newtheorem{theorem}{Theorem}[section]
\newtheorem{prop}[theorem]{Proposition}
\newtheorem{lemma}[theorem]{Lemma}

\newtheorem{definition}[theorem]{Definition}

\newtheorem{exercise}[theorem]{Exercise}
\newtheorem{example}[theorem]{Example}

\begin{document}

\title{An introduction to quantum symmetries}

\author{Christian Voigt}
\address{School of Mathematics and Statistics \\
         University of Glasgow \\
         University Place \\
         Glasgow G12 8QQ \\
         United Kingdom} 
\email{christian.voigt@glasgow.ac.uk}



\maketitle

\begin{abstract}
These notes are an introduction to the theory of quantum symmetries of finite and infinite sets, graphs, and locally compact spaces. 
\end{abstract}

\section{Preface} 

The study of \emph{quantum symmetries} is a line of research originating from work of Wang 
\cite{WANG_quantumsymmetry}, who defined quantum symmetry groups of finite sets, and more generally of finite dimensional $ C^* $-algebras, using Woronowicz's theory of compact quantum groups \cite{WORONOWICZ_compactquantumgroups}. The resulting quantum groups have been studied intensively throughout the past three decades, and this has uncovered intriguing links with free probability, subfactors, tensor categories, combinatorics, and quantum information theory.  

The aim of these notes is to give an introduction to this topic, without assuming any prior familiarity with quantum symmetries or quantum groups. 
However, some working knowledge of Hilbert space operators and $ C^* $-algebras will be assumed, as for instance covered in \cite{MURPHY_book}. 

Our starting point is the concept of a quantum permutation of a finite set, and we explain how such quantum permutations 
naturally form a $ C^* $-tensor category. This is complementary to the original approach by Wang, in the sense that it shifts focus from the $ C^* $-algebras introduced in \cite{WANG_quantumsymmetry} to their categories of representations. From $ C^* $-tensor categories it is only 
a small step towards quantum groups, and we discuss both discrete quantum groups and the closely related compact quantum groups already mentioned above. However, it is important to point out that one genuinely obtains \emph{discrete} quantum groups in our approach, in contrast to the \emph{compact} ones in Wang's theory. 
We will mostly work directly with quantum permutations anyway, and our main aim is to show that this perspective generalises naturally to infinite sets and graphs, and even to the realm of locally compact spaces. In all these cases, the  constructions will be illustrated with examples as we go along. 

What these notes do \emph{not} touch upon is the combinatorial and analytical study of Wang's quantum permutation groups and their quantum subgroups, in particular those arising as quantum automorphism groups of finite graphs and related structures. Fortunately, there is already an extensive literature on this topic, including the excellent textbooks \cite{BANICA_quantumpermutationgroups}, \cite{FRESLON_compactquantumgroupscombinatorics}. We refer the reader to these resources and the references therein for further study. 

These are notes for the CIMPA School ``$ K $-theory and Operator Algebras'' which took place 28 July - 1 August 2025 in La Plata (Argentina). Special thanks go to G. Corti\~nas and G. Tartaglia for their excellent organisation, and I would also like to use this opportunity to thank the participants of the school for interesting discussions during and after lectures.

\section{Introduction} \label{sec_introduction}

\subsection{What is quantum symmetry?} 

In order to motivate the main idea behind quantum symmetries we start with the most fundamental example, namely quantum permutations of a finite set. 

Recall that a \emph{permutation} of $ n $ points is a bijective map $ \sigma: [n] \rightarrow [n] $, where we write $ [n] = \{1,\dots, n\} $. We can encode such a permutation in terms of the $ n \times n $-matrix $ u^\sigma $ given by 
$$
u^{\sigma}_{x,y} = 
\begin{cases} 
1 & x = \sigma(y) \\
0 & \text{else.}
\end{cases} 
$$
A matrix of this form is called a \emph{permutation matrix}. It follows directly from the definition that a matrix $ u \in M_n(\mathbb{C}) $ is a permutation 
matrix iff all its entries are either $ 0 $ or $ 1 $, and each row and each column of $ u $ contains precisely one entry $ 1 $. 
This can equivalently be phrased as follows. 

\begin{lemma} \label{permmatrices}
A matrix $ u = (u_{x,y}) \in M_n(\mathbb{C}) $ is a permutation matrix if and only if the following conditions are satisfied.  
\begin{bnum} 
\item[$\bullet$] $ u_{x,y}^2 = u_{x,y} $ for all $ 1 \leq x,y \leq n $.
\item[$\bullet$] We have 
$$ 
\sum_{t = 1}^n u_{x,t} = 1 = \sum_{t = 1}^n u_{t,y} 
$$
for all $ 1 \leq x,y \leq n $. 
\end{bnum}
\end{lemma} 

The key idea, due to Wang \cite{WANG_quantumsymmetry}, is to study more general ``permutation matrices'', by considering the relations in Lemma \ref{permmatrices} not 
over $ \mathbb{C} $, but for matrices with entries in (noncommutative) $ C^* $-algebras. In the sequel, the following definition will play a central role. 

\begin{definition} \label{defqperm} 
A \emph{quantum permutation} of $ [n] = \{1, \dots, n\} $ is a pair $ (\Hc, u) $ consisting of a Hilbert space $ \Hc $ and an $ n \times n $-matrix $ u = (u_{x,y}) $ of 
elements $ u_{x,y} \in B(\Hc) $ such that 
\begin{bnum} 
\item[$\bullet$] $ u_{x,y}^2 = u_{x,y} = u_{x,y}^* $ for all $ 1 \leq x,y \leq n $.
\item[$\bullet$] We have 
$$ 
\sum_{t = 1}^n u_{x,t} = 1 = \sum_{t = 1}^n u_{t,y} 
$$
for all $ 1 \leq x,y\leq n $. 
\end{bnum}
Matrices $ u = (u_{x,y}) \in M_n(B(\Hc)) $ satisfying these conditions are also called \emph{magic unitaries}. 
\end{definition} 

Here we write $ T^* $ for the adjoint of $ T \in B(\Hc) $, and we note that the first condition in 
Definition \ref{defqperm} is saying that all entries of a magic unitary are projections. Using Exercise \ref{exI1} below, it is not hard to check that a magic unitary is indeed a unitary element of $ M_n(B(\Hc)) $. 

Note also that the requirement $ u_{x,y} = u_{x,y}^* $ is not part of Lemma \ref{permmatrices}, but it could have been added there since the elements $ 0,1 \in \mathbb{C} $ are clearly self-adjoint. We thus arrive at the following basic fact. 

\begin{lemma} \label{qpermdimone}
Quantum permutations of $ [n] $ whose underlying Hilbert space is $ \mathbb{C} $ are the same thing as permutations of $ [n] $. 
\end{lemma} 

By the \emph{dimension} of a quantum permutation $ (\Hc,u) $ we shall mean the dimension of the underlying Hilbert space $ \Hc $. In view of Lemma \ref{qpermdimone}, 
we will be mainly interested in quantum permutations on Hilbert spaces of dimension strictly greater than one. 

Let us introduce some more terminology. We shall say that a quantum permutation $ (\Hc, u) $ is \emph{classical} if all matrix entries $ u_{x,y} $ mutually commute. 
Clearly, every one-dimensional quantum permutation is classical, which means that quantum permutations arising from classical permutations are indeed classical. 
Not every classical quantum permutation is of this form, but we will explain later that there is indeed a close link between classical quantum permutations 
and (classical) permutations. 

Let us now describe the structure of quantum permutations for small values of $ n $. 

\begin{lemma} \label{quantumpermutationssmalln}
Every quantum permutation of $ [n] $ for $ n = 1,2,3 $ is classical. 
\end{lemma} 

\begin{proof} 
This is obvious for $ n = 1 $. For $ n = 2 $ notice that every quantum permutation $ (\Hc, u) $ must be of the form 
$$
u =
\begin{pmatrix} 
p & 1 - p \\
1 - p & p 
\end{pmatrix} 
$$
for some projection $ p \in B(\Hc) $. 

For $ n = 3 $ it is enough to show that the matrix entries $ u_{i,j} $ and $ u_{k,l} $ commute provided $ i \neq k $ and $ j \neq l $. 
Consider for instance $ u_{11} $ and $ u_{22} $. We get 
$$
u_{1,1} u_{2,2} u_{1,3} = u_{1,1} (1 - u_{2,1} - u_{2,3}) u_{1,3} = 0,  
$$
which implies $ u_{1,1} u_{2,2} = u_{1,1} u_{2,2} (u_{1,1} + u_{1,2} + u_{1,3}) = u_{1,1} u_{2,2} u_{1,1} $. 
This yields 
$$
u_{1,1} u_{2,2} = u_{1,1} u_{2,2} u_{1,1} = (u_{1,1} u_{2,2} u_{1,1})^* = (u_{1,1} u_{2,2})^* = u_{2,2} u_{1,1} 
$$
as required. The remaining cases can be verified in a similar way, see also Exercise \ref{exI3}. 
\end{proof} 

Lemma \ref{quantumpermutationssmalln} is basically saying that quantum permutations of sets of cardinality at most $ 3 $ are not very interesting. In contrast, as soon as $ n \geq 4 $ one can find non-classical quantum permutations. For instance, given arbitrary projections $ p,q \in B(\Hc) $ the matrix 
$$
\begin{pmatrix} 
p & 1 - p & 0 & 0 \\
1 - p & p & 0 & 0 \\
0 & 0 & q & 1 - q \\
0 & 0 & 1 - q & q 
\end{pmatrix} 
$$
defines a quantum permutation of four points. 

In order to understand the case $ n = 4 $ in more detail, let us discuss a method which produces further examples of non-classical quantum permutations, also for higher values of $ n $, compare \cite[Chapter 16a]{BANICA_quantumpermutationgroups}.  
Fix $ N \in \mathbb{N} $ and consider an abelian group $ A $ of order $ N $, for example $ A = \mathbb{Z}/N \mathbb{Z} $. Identify $ A $ with the Pontrjagin dual group $ \widehat{A} $ and let $ [ \;, \; ]: A \times A \rightarrow U(1)  $ be the map induced by the canonical pairing between $ A $ and $ \widehat{A} $. For instance, in the case $ A = \mathbb{Z}/N \mathbb{Z} $ we get 
$$
[k, l] = e^{2 \pi i kl/N} 
$$
for $ k,l \in \mathbb{Z}/N \mathbb{Z} $. Consider the Hilbert space $ \Hc = \mathbb{C}^N $ and label its standard basis vectors $ e_a $ by the group elements $ a \in A $. 
We define the \emph{Weyl matrices} $ W_{a,b} $ by 
$$
W_{a,b} e_c = [a,c] e_{b + c} 
$$
for $ a,b,c \in A $. Moreover let $ \tr: M_N(\mathbb{C}) \rightarrow \mathbb{C} $ be the normalised trace, so that $ \tr(1) = 1 $. 

\begin{lemma} \label{weyllemma}
The Weyl matrices are unitary and satisfy the following relations. 
\begin{bnum} 
\item[a)] $ W_{a,b}^* = [a,b] W_{-a, -b} $ 
\item[b)] $ W_{a,b} W_{c,d} = [a,d] W_{a + c, b + d} $
\item[c)] $ W_{a,b}^* W_{c,d} = [a, b - d] W_{c - a, d - b} $
\item[d)] $ \tr(W_{a,b}^* W_{c,d}) = \delta_{a,c} \delta_{b,d} $
\end{bnum} 
\end{lemma} 

\begin{proof} 
$ a) $ We compute 
\begin{align*}
\bra W_{a,b}^* e_c, e_d \ket = \bra e_c, [a,d] e_{b + d} \ket &= [a,d] \delta_{c, b + d} \\
&= [a, c - b] \delta_{c - b, d} \\ 
&= [a, -b] [a, c] \bra e_{c - b}, e_d \ket = \bra [a,b] W_{-a,-b} e_c, e_d \ket 
\end{align*}
for all $ c,d \in A $, and this yields the claim. \\ 
$ b) $ We calculate 
$$
W_{a,b} W_{c,d} e_f = [c, f] W_{a,b} e_{d + f} = [c, f] [a, d + f] e_{b + d + f} = [a,d] W_{a + c, b + d} e_f 
$$
as required. \\
$ c) $ follows by combinining $ a) $ and $ b) $. Taking $ a = c, b = d $ shows in particular that $ W_{a,b} $ is unitary. \\
$ d) $ Using $ c) $ we compute 
\begin{align*}
\tr(W_{a,b}^* W_{c,d}) &= \frac{1}{N} \sum_{f \in A} \bra e_f, W_{a,b}^* W_{c,d} e_f \ket \\
&= \frac{1}{N} \sum_{f \in A} [a, b - d] \bra e_f, [c - a, f] e_{f + d - b} \ket \\
&= \delta_{b,d} \frac{1}{N} \sum_{f \in A} [a, 0] \bra e_f, [c - a, f] e_f \ket \\
&= \delta_{b,d} \frac{1}{N} \sum_{f \in A} [c - a, f] = \delta_{b,d} \delta_{a,c} 
\end{align*}
as claimed. 
\end{proof} 

We remark that the Weyl matrices form a \emph{unitary error basis}. That is, they form a basis of $ M_N(\mathbb{C}) $ consisting of unitary matrices which 
are mutually orthogonal with respect to the Hilbert-Schmidt inner product 
$$
\bra S,T \ket = \tr(S^* T).  
$$
In quantum information theory, unitary error bases are closely related to ``tight'' quantum teleportation and superdense coding schemes, see \cite{WERNER_teleportation}. 

\begin{prop} \label{weylquantumpermutation}
Let $ A $ be an abelian group of order $ N $. Set $ n = N^2 $ and identify $ M_N(\mathbb{C}) \cong \mathbb{C}^n $. For every unitary matrix $ g \in U(N) $ we obtain a quantum permutation $ \pi^g = (\mathbb{C}^n, u^g) $ of $ n $ points by setting 
$$
u^g_{(a,b),(c,d)} = |W_{a,b} g W_{c,d}^* \ket \bra W_{a,b} g W_{c,d}^*| 
$$
for $ (a,b), (c,d) \in A \times A \cong [n] $, that is, $ u^g_{(a,b),(c,d)} $ is the orthogonal projection onto the linear subspace spanned by $ W_{a,b} g W_{c,d}^* $. 
\end{prop} 

\begin{proof} 
It follows from Lemma \ref{weyllemma} c), d) that $ W_{a,b} g W_{c,d}^* $ is a unit vector in $ M_N(\mathbb{C}) = \mathbb{C}^n $. Hence each $ u^g_{(a,b),(c,d)} $ is a projection. 
Moreover, for fixed $ a,b $ the matrices $ W_{a,b} g W_{c,d}^* $ for $ c, d \in A $ form an orthonormal basis of $ M_N(\mathbb{C}) $ since 
$$
\tr((W_{a,b} g W_{c_1,d_1}^*)^* W_{a,b} g W_{c_2,d_2}^*) 
= \tr(W_{c_1,d_1} W_{c_2,d_2}^*) = \delta_{c_1, c_2} \delta_{d_1, d_2} 
$$
by Lemma \ref{weyllemma} d). 
It follows that $ \sum_{c,d} u^g_{(a,b), (c,d)} = 1 $. Similarly, one checks the relation $ \sum_{a,b} u^g_{(a,b), (c,d)} = 1 $ for all $ c,d \in A $. 
\end{proof} 

Note that the magic unitaries constructed in Proposition \ref{weylquantumpermutation} only depend on the class of the unitary $ g $ in the projective unitary group $ PU(N) = U(N)/U(1) $. In other words, we can rescale $ g $ by an arbitrary phase without changing the magic unitary associated to it. 

In the special case $ A = \mathbb{Z}/2 \mathbb{Z} $, the following result from \cite{BANICA_COLLINS_pauli}, \cite{BANICA_BICHON_fourpoints} shows that the magic unitaries in Proposition \ref{weylquantumpermutation} yield a complete description of the irreducible quantum 
permutations of four points, see also \cite[Lemma 5.4]{VOIGT_infinitequantumpermutations}. Here a quantum permutation $ (\Hc, u) $ is called \emph{irreducible} if the only elements of $ B(\Hc) $ 
commuting with all entries $ u_{x,y} $ of $ u $ are multiplies of the identity. 

\begin{theorem}[Banica-Collins 2008, Banica-Bichon 2009] \label{BBfourpoints}
The irreducible quantum permutations of the set $ \{1,2,3,4\} \cong \mathbb{Z}/2\mathbb{Z} \times \mathbb{Z}/2 \mathbb{Z} $ have dimension $ 1,2 $ or $ 4 $, and can all be represented as direct summands of Weyl 
quantum permutations of the form $ \pi^g = (\mathbb{C}^4, u^g) $ for $ g \in PU(2) \cong SO(3) $, where 
$$
u^g_{(a,b),(c,d)} = |W_{a,b} g W_{c,d}^* \ket \bra W_{a,b} g W_{c,d}^*|.
$$
Moreover, two irreducible quantum permutations $ \pi^g, \pi^h $ are equivalent iff $ g = \alpha h \beta $ for some elements $ \alpha, \beta $ in the diagonal subgroup of $ SO(3) $. 
\end{theorem} 

We shall not give a proof of Theorem \ref{BBfourpoints} here. In fact, even the statement of the theorem involves some terminology which we have not yet explained, such as \emph{direct summands} of quantum permutations and \emph{equivalence}. These notions will be discussed in section \ref{ch2}.

The construction in Theorem \ref{BBfourpoints} provides a \emph{matrix model} for magic unitary $ 4 \times 4 $-matrices. 
We point the reader to Part IV of \cite{BANICA_quantumpermutationgroups} for more information on matrix models for quantum permutations.

\subsection{Quantum automorphisms of graphs} \label{sec_qutgraph}

For higher values of $ n $, the structure of general quantum permutations of $ n $ points is hard to describe concretely. It is therefore natural to look for quantum permutations with additional properties. A prominent example of such ``structured'' quantum permutations arises from quantum symmetries of graphs, see \cite{BANICA_quthomogeneousgraphs}, \cite{BANICA_BICHON_ordereleven}, \cite{SCHMIDT_distancetransitive}.

In the sequel, by a \emph{graph} we mean a simple, undirected graph without self-loops. 
More precisely, a graph $ X = (V_X, E_X) $ consists of a set $ V_X $ of vertices and a set $ E_X \subset V_X \times V_X $ of edges. The conditions on the edge set $ E_X $ are that $ (x,y) \in E_X $ iff $ (y,x) \in E_X $ and $ (x,x) \notin E_X $ for all $ x \in V_X $. 
We note that the structure of $ X $ is completely determined by the set $ V_X $ together with the \emph{adjacency matrix} $ A_X $, which is the $ V_X \times V_X $-matrix with entries 
$$
(A_X)_{x,y} = 
\begin{cases} 
1 \text{ if } (x,y) \in E_X \\
0 \text{ else.} 
\end{cases}
$$
The graph $ X $ is called \emph{finite} if $ V_X $ is a finite set, and $ E_X $ is clearly finite as well in this case since we do not allow multiple edges.

If $ X = (V_X, E_X), Y = (V_Y, E_Y)$ are graphs then a map $ \sigma: V_X \rightarrow V_Y $ is called a \emph{graph homomorphism} from $ X $ to $ Y $ if $ (x,y) \in E_X $ implies $ (\sigma(x), \sigma(y)) \in E_Y $ for all $ x,y \in V_X $. By definition, a \emph{graph isomorphism} is a bijective graph homomorphism. In this case, the condition that adjacency between vertices is preserved can be phrased equivalently as saying that the matrix $ u^\sigma $ associated to $ \sigma $, defined by the formula in the discussion before Lemmma \ref{permmatrices}, satisfies $ A_Y u^\sigma = u^\sigma A_X $. For the moment we are only interested in finite graphs, but we note that the matrix products $ A_Y u^\sigma $ and $ u^\sigma A_X $ make sense in general. We will say more about infinite graphs in section \ref{sec_infinite}. 

A \emph{graph automorphism} of $ X $ is a graph isomorphism $ \sigma $ from $ X $ to itself and, according to the above, the associated permutation matrix then satisfies $ A_X u^\sigma = u^\sigma A_X $. 
With this in mind, we arrive naturally at the following definition of a \emph{quantum automorphism}, compare \cite{BANICA_quthomogeneousgraphs}. 

\begin{definition} \label{defquatfinitegraph}
Let $ X = (V_X, E_X) $ be a finite graph. A quantum automorphism of $ X $ is a quantum permutation $ (\Hc,u) $ of $ V_X $ such that $ A_X u = u A_X $.  
\end{definition} 

In this definition, the formula $ A_X u = u A_X $ is understood as an equality in $ M_n(B(\Hc)) $ where $ n = |V_X| $, and the entries of $ A_X $ are viewed as scalar multiples of the identity in $ B(\Hc) $ in the obvious way. Extending Definition \ref{defquatfinitegraph}, we also say that two finite graphs $ X, Y $ are \emph{quantum isomorphic} if they both have the same number $ n = |V_X| = |V_Y| $ of vertices and there exists a quantum permutation $ (\Hc, u) $ of $ [n] $ such that $ A_Y u = u A_X $. Here we tacitly identify $ V_X = [n] = V_Y $. 

\begin{lemma} \label{adjacencyrelations}
Let $ X = (V_X, E_X) $ be a finite graph with adjacency matrix $ A_X $ and let $ (\Hc, u) $ be a quantum permutation of the vertex set $ V_X $. Then the following conditions are equivalent. 
\begin{bnum} 
\item[a)] $ (\Hc,u) $ is a quantum automorphism of $ X $. 
\item[b)] For all $ i,j,k,l \in V_X $ we have 
$
u_{i,j} u_{k,l} = 0 = u_{k,l} u_{i,j} 
$
whenever $ (A_X)_{i,k} \neq (A_X)_{j,l} $. 
\end{bnum} 
\end{lemma} 

\begin{proof} 
Note first that the condition $ A_X u = u A_X $ means
\begin{align*}
\sum_t (A_X)_{r,t} u_{t,s} = (A_X u)_{r,s} = (u A_X)_{r,s} = \sum_t (A_X)_{t,s} u_{r,t}
\end{align*}
for all $ r,s \in V_X $. 

$ a) \Rightarrow b) $ Assume that $ i,k,j,l \in V_X $ are such that $ (A_X)_{i,k} = 0, (A_X)_{j,l} = 1 $. 
Consider $ r = i, s = l $ in the above formula and multiply by $ u_{k,l} $ from the left and from the right to get
$$
\sum_t (A_X)_{t,l} u_{k,l} u_{i,t} u_{k,l} = \sum_t (A_X)_{i,t} u_{k,l} u_{t,l} u_{k,l} = (A_X)_{i,k} u_{k,l} = 0.  
$$
Since the elements $ u_{k,l} u_{i,t} u_{k,l} $ are positive and $ (A_X)_{j,l} = 1 $ we conclude that $ u_{k,l} u_{i,j} u_{k,l} $ is zero. 
But this means $ u_{i,j} u_{k,l} = 0 $ since $ (u_{i,j} u_{k,l})^* u_{i,j} u_{k,l} = u_{k,l} u_{i,j} u_{k,l} $. 

If $ (A_X)_{i,k} = 1 $ and $ (A_X)_{j,l} = 0 $ we can run a similar argument. Consider again $ r = i, s = l $ and multiply the above relation by $ u_{i,j} $ 
from the left and right. We get 
$$
0 = (A_X)_{j,l} u_{i,j} = \sum_t (A_X)_{t,l} u_{i,j} u_{i,t} u_{i,j} = \sum_t (A_X)_{i,t} u_{i,j} u_{t,l} u_{i,j} 
$$
in this case, and deduce $ u_{i,j} u_{k,l} = 0 $. 

$ b) \Rightarrow a) $ Assume that $ u_{i,j} u_{k,l} = 0 = u_{k,l} u_{i,j} $ whenever $ (A_X)_{i,k} \neq (A_X)_{j,l} $. Then 
\begin{align*}
\sum_t (A_X)_{t,s} u_{r,t} = \sum_{t,p} (A_X)_{t,s} u_{r,t} u_{p,s} 
&= \sum_{t,p \mid (A_X)_{t,s} = 1} u_{r,t} u_{p,s} \\
&= \sum_{t, p \mid (A_X)_{r,p} = 1} u_{r,t} u_{p,s} \\
&= \sum_{t, p} (A_X)_{r,p} u_{r,t} u_{p,s} = \sum_p (A_X)_{r,p} u_{p,s} 
\end{align*}
for all $ r,s \in V_X $ as required. 
\end{proof} 

It is in general a difficult problem to determine the quantum automorphisms of a given graph $ X $. We say that $ X $ has \emph{no quantum symmetry} if every 
irreducible quantum automorphism of $ X $ is classical, that is, of dimension one. Otherwise we say that $ X $ \emph{has quantum symmetry}.

The existence of quantum symmetry is completely understood only for graphs with a small number of vertices, compare \cite{BANICA_BICHON_ordereleven}.

\subsection{The graph isomorphism game} 

The study of quantum symmetries of graphs has an intriguing link to quantum information theory through 
the \emph{graph isomorphism game} \cite{ATSERIAS_MANCINSKA_ROBERSON_SAMAL_SEVERINI_VARVITSIOTIS_quantumgraphisomorphism}, \cite{LUPINI_MANCINSKA_ROBERSON_quantumpermutationgroups}. This 
is an example of a \emph{non-local game}, that is, a game played by two cooperating players, Alice and Bob, against a referee. In each round of the game, the referee sends both Alice and Bob an input, and they have to return an output independently of each other. The players are allowed to agree a common strategy at the outset, but they are not allowed to communicate during the game.  Whether Alice and Bob win the game is determined by a verifier function which, together with the input and output sets, is known to all parties beforehand. 

In the case of the graph isomorphism game the input data consists of two finite graphs $ X, Y $, and the task for Alice and Bob is to 
convince the referee that they can produce a graph isomorphism between $ X $ and $ Y $. 
The input and output sets of this game are the same, given by the (disjoint) union $ V = V_X \cup V_Y $ of the underlying vertex sets of $ X, Y $. In each round of the game, the referee sends vertices $ x,y \in V $ to Alice and Bob, respectively, which may or may not come from the same graph. 
Alice and Bob have to return vertices $ a,b \in V $, and these should satisfy the following conditions. Firstly, if $ x $ is from $ V_X $ then $ a $ has to be from $ V_Y $ and vice versa, and the same for $ y $ and $ b $. If this is satisfied, then two of the four vertices $ a,b, x, y $ are from $ X $ and two are from $ Y $. The second condition is that the pair from $ X $ must satisfy the same adjacency relation as the pair from $ Y $. For instance,  
if $ x,y $ both come from $ Y $ 
then in order for Alice and Bob to win the game their answers $ a, b $ must come from $ X $ and be equal/adjacent/non-adjacent iff $ x, y $ are. 
The game can be repeated many times, in order to reduce the likelihood that Alice and Bob win by chance, but it is not required that the players give the same answers to the same questions in subsequent rounds of the game. 

As already discussed above, Alice and Bob are not allowed to communicate during the game. This means that they do not know which vertex their partner is being sent, or which vertex 
the other is returning. They are allowed, however, to agree a \emph{strategy} in advance. In mathematical terms, a strategy is encoded by a 
\emph{correlation}, or conditional probability density. We will first give the following general definition, and then return to the case of the graph isomorphism game. 

\begin{definition} \label{defcorrelation}
Fix finite sets $ I $ and $ O $ of cardinality $ |I| = n $ and $ |O| = k $, respectively. 
A correlation is a map $ p: O \times O \times I \times I \rightarrow [0,1] $ 
such that $ \sum_{a,b \in O} p(a,b|x,y) = 1 $ for all $ x, y \in I $. 
\end{definition} 

There are various classes of correlations, determined by the way they can be obtained. By definition, if there is a map $ f: I \times I \rightarrow O \times O $ such that 
$$ 
p(a,b|x,y) = 
\begin{cases}
1 & \text{if } f(x,y) = (a,b) \\
0 & \text{else}
\end{cases}
$$
then $ p $ is called a \emph{deterministic} correlation. 

Another important class of correlations is obtained in the following way. 
Let $ \Hc $ be a Hilbert space and let $ \psi \in \Hc $ be a unit vector. Assume moreover that $ A^a_x, B^b_y $ 
are positive operators in $ B(\Hc) $ such that $ \sum_a A^a_x = 1 = \sum_b B^b_y $ for all $ x, y \in I $, and that the operators $ A^a_x $ and $ B^b_y $ commute 
for all $ a,b \in O $ and $ x,y \in I $. Then we obtain a correlation $ p $ by setting 
$$
p(a,b|x,y) = \bra \psi, A^a_x B^b_y \psi \ket.  
$$
Correlations of this form are called  \emph{quantum commuting}. 

Recall that in the case of the graph isomorphism game we have $ I = O = V_X \cup V_Y $. 
A correlation $ p $ is called \emph{perfect} or, equivalently, a \emph{winning strategy} for the graph isomorphism game if $ p(a,b|x,y) \neq 0 $ implies that $ a,b $ is a valid answer to the question $ x,y $. For instance, if $ \sigma: X \rightarrow Y $ is a graph isomorphism then the deterministic correlation $ I \times I \rightarrow O \times O $ 
obtained by extending $ \sigma $ and $ \sigma^{-1} $ in the obvious way is a winning strategy. 

The following result from \cite{ATSERIAS_MANCINSKA_ROBERSON_SAMAL_SEVERINI_VARVITSIOTIS_quantumgraphisomorphism}, \cite{LUPINI_MANCINSKA_ROBERSON_quantumpermutationgroups} shows that quantum isomorphism of graphs can be expressed in terms of winning strategies for the graph isomorphism game. 

\begin{theorem} \label{qimagicchar}
Two finite graphs $ X, Y $ are quantum isomorphic iff the graph isomorphism game for $ X $ and $ Y $ admits a winning quantum commuting strategy. 
\end{theorem} 

It is an intriguing fact that Alice and Bob can win the $ X $-$ Y $-graph isomorphism game with a perfect quantum commuting strategy even if the graphs $ X $ and $ Y $ are not isomorphic. In other words, there exist pairs of graphs which are quantum isomorphic but not isomorphic  \cite{ATSERIAS_MANCINSKA_ROBERSON_SAMAL_SEVERINI_VARVITSIOTIS_quantumgraphisomorphism}. 

Remarkably, it turns out that quantum isomorphism of graphs can be expressed purely in graph theoretical terms. This is the content of the 
following celebrated result of Man\v{c}inska-Roberson \cite{MANCINSKA_ROBERSON_quantumisoplanar}.  

\begin{theorem} \label{qiMRchar} 
Let $ X, Y $ be finite graphs. Then the following conditions are equivalent. 
\begin{bnum}
\item[a)] $ X $ and $ Y $ are quantum isomorphic. 
\item[b)] For every planar graph $ P $ there exists a bijection $ \Hom(P, X) \cong \Hom(P,Y) $. 
\end{bnum}
\end{theorem}

Here, for graphs $ P, Q $, we write $ \Hom(P,Q) $ for the set of all (classical) graph homomorphisms from $ P $ to $ Q $. 
Theorem \ref{qiMRchar} shows in particular that if $ X $ and $ Y $ are quantum isomorphic then they do not only have the same number of vertices, 
but also the same number of edges.

We note that if $ X, Y $ are quantum isomorphic then the proof of Theorem  \ref{qiMRchar} does not provide explicit maps realising the bijections $ \Hom(P, X) \cong \Hom(P,Y) $, it rather shows that the cardinalities of the sets $ \Hom(P, X) $ and $ \Hom(P, Y) $ are the same. Let us also remark that the problem of determining whether two graphs are quantum isomorphic is known to be undecidable \cite{ATSERIAS_MANCINSKA_ROBERSON_SAMAL_SEVERINI_VARVITSIOTIS_quantumgraphisomorphism}. 

\subsection{Exercises}

\begin{exercise} \label{exI1}
Let $ p_1, \dots, p_n $ be projections in a unital $ C^* $-algebra such that $ p_1 + \cdots + p_n = 1 $. Show that $ p_i p_j = 0 $ for $ i \neq j $. 
\end{exercise} 

\begin{exercise}
Let $ \sigma = (\Hc,u) $ be a quantum permutation of four points. What is the smallest number of mutually commuting matrix entries $ u_{i,j} $ needed to ensure that $ \sigma $ is classical?
\end{exercise}

\begin{exercise} \label{exI3}
Let $ \sigma, \tau $ be classical permutations and let $ (\Hc,u) $ be a quantum permutation of $ [n] $. Show that $ v = (v_{i,j}) $ defined by $ v_{i,j} = u_{\sigma(i), \tau(j)} $ is again a quantum permutation. Use this to complete the proof of Lemma \ref{quantumpermutationssmalln} without any additional computations. 
\end{exercise}

\begin{exercise}
For the group $ A = \mathbb{Z}/2\mathbb{Z} $, find unitaries $ g,h \in U(2) $ such that the Weyl quantum permutation $ \pi^g $ is classical and $ \pi^h $ is non-classical.  
\end{exercise}

\section{Tensor categories and quantum groups} \label{ch2}

In this chapter we discuss some background material from the theory of tensor categories and quantum groups. These two concepts are related through the Tannaka-Krein formalism, and we refer to \cite{NESHVEYEV_TUSET_compactquantumgroups} for more information.
We will focus here on tensor categories associated with quantum permutations. 

\subsection{Constructions with quantum permutations} 

Recall from Definition \ref{defqperm} that a quantum permutation $ (\Hc,u) $ of a finite set $ X $ consists of a Hilbert space $ \Hc $ and a matrix $ u = (u_{x,y}) $ of projections $ u_{x,y} \in B(\Hc) $ forming a magic unitary. 

\begin{definition} \label{def_constructions}
Let $ X $ be a finite set and let $ \sigma = (\Hc_\sigma, u^\sigma) $ and $ \tau = (\Hc_\tau, u^\tau) $ be quantum permutations of $ X $. 
\begin{bnum}
\item[a)] The direct sum $ \sigma \oplus \tau = (\Hc_\sigma \oplus \Hc_\tau, u^\sigma \oplus u^\tau) $ is defined by taking the direct sum of the underlying 
Hilbert spaces and setting
$$ 
(u^\sigma \oplus u^\tau)_{x,y} = u^\sigma_{x,y} \oplus u^\tau_{x,y} 
$$ 
for $ x,y \in X $. 
\item[b)] The tensor product $ \sigma \otimes \tau = (\Hc_\sigma \otimes \Hc_\tau, u^\sigma \tp u^\tau) $ is defined by taking the Hilbert space tensor product 
of the underlying Hilbert spaces and setting
$$ 
(u^\sigma \tp u^\tau)_{x,y} = \sum_{t \in X} u^\sigma_{x,t} \otimes u^\tau_{t,y} 
$$ 
for $ x,y \in X $. 
\item[c)] The conjugate $ \overline{\sigma} = (\Hc_{\overline{\sigma}}, u^{\overline{\sigma}}) $ of $ \sigma = (\Hc_\sigma, u^\sigma) $ 
is defined by letting $ \Hc_{\overline{\sigma}} $ be the conjugate Hilbert space of $ \Hc_\sigma $ and taking $ u^{\overline{\sigma}} = (u^{\overline{\sigma}}_{x,y}) $ 
determined by $ u^{\overline{\sigma}}_{x,y}(\overline{\xi}) = \overline{u^\sigma_{y,x}(\xi)} $ for $ \xi \in \Hc_\sigma $ and $ x,y \in X $. 
\end{bnum}
\end{definition} 

We note that the definition of direct sums can be extended to arbitrary families of quantum permutations, possibly infinite, in a straightforward way. 

The constructions in Definition \ref{def_constructions} allow us to obtain new quantum permutations out of given ones. 

\begin{lemma} \label{tensorstructure}
Direct sums, tensor products, and conjugates of quantum permutations are again quantum permutations. 
\end{lemma} 

\begin{proof} 
The case of direct sums is immediate. For tensor products, note that  
$$
\sum_{x \in X} (u^\sigma \tp u^\tau)_{x,y} = \sum_{x, t \in X} u^\sigma_{x,t} \otimes u^\tau_{t,y} = \sum_{t \in X} 1 \otimes u^\tau_{t,y} = 1 \otimes 1 
$$
for all $ y \in X $, and similarly $ \sum_{y \in X} (u^\sigma \tp u^\tau)_{x,y} = 1 $. Hence $ u^\sigma \tp u^\tau $ is a magic unitary as required. 
The case of conjugates is again obvious. 
\end{proof} 

There is also a natural notion of \emph{morphisms} between quantum permutations which we introduce next. 

\begin{definition} \label{defintertwiner}
Let $ X $ be a finite set and let $ \sigma = (\Hc_\sigma, u^\sigma) $ and $ \tau = (\Hc_\tau, u^\tau) $ be quantum permutations of $ X $. 
An intertwiner from $ \sigma $ to $ \tau $ is a bounded linear operator $ T: \Hc_\sigma \rightarrow \Hc_\tau $ such that $ T u^\sigma_{x,y} = u^\tau_{x,y} T $ 
for all $ x, y \in X $. 
\end{definition} 

Using this terminology, we can rephrase the condition that a quantum permutation $ \sigma = (\Hc, u) $ is irreducible, introduced in the previous chapter, as saying that the only intertwiners between $ \sigma $ and itself are multiples of the identity. By definition, a \emph{subobject} of a quantum permutation $ \tau = (\Hc_\tau, u^\tau) $ is a quantum permutation $ \sigma = (\Hc_\sigma, u^\sigma) $ together with an isometric intertwiner from $ \sigma $ to $ \tau $. In this case we can identify $ \Hc_\sigma $ with a closed subspace of $ \Hc_\tau $, in such a way that $ u^\sigma_{x,y} $ is obtained by restriction from $ u^\tau_{x,y} $ for all $ x,y \in X $. 

Let us come back to the question in which sense classical quantum permutations are related to (ordinary) classical permutations. 
If $ \sigma_1, \dots, \sigma_n $ are one-dimensional quantum permutations, so associated to classical permutations of $ X $ by Lemma \ref{qpermdimone}, 
then their direct sum $ \sigma_1 \oplus \cdots \oplus \sigma_n $ is a classical quantum permutation. Moreover, every finite dimensional classical quantum permutation is of this form. In fact, it is not hard to show that \emph{every} classical quantum permutation is equivalent to a (possibly infinite) direct sum of one-dimensional quantum permutations, see Exercise \ref{ex22}. In summary, not every classical quantum permutation can be described by a (single) classical permutation. However, every \emph{irreducible} classical quantum permutation 
is one-dimensional, and associated to a uniquely determined classical permutation.

\subsection{$ C^* $-tensor categories} 

The operations with quantum permutations described in the previous section correspond to the main structural ingredients featuring in a \emph{$ C^* $-tensor category}. In this section we discuss the definition of $ C^* $-tensor categories and state some general facts about them. For a more comprehensive treatment we refer to \cite{NESHVEYEV_TUSET_compactquantumgroups}. 

\subsubsection{$ C^* $-categories} 

Recall that a \emph{category} $ \BC $ is given by a class of objects together with morphism sets $ \BC(X,Y) = \Mor_\BC(X,Y) $ for each pair of objects $ X, Y \in \BC $. By a \emph{$ * $-category} we shall mean a category $ \BC $ with the following structure and properties. Firstly, all morphism spaces $ \BC(X,Y) $ for objects $ X, Y \in \BC $ are complex vector spaces and 
the composition maps $ \BC(X,Y) \times \BC(Y, Z) \rightarrow \BC(X, Z), (f,g) \mapsto g \circ f $ are bilinear. Secondly, the category is equipped with an antilinear involutive contravariant endofunctor $ *: \BC \rightarrow \BC $ which is the identity on objects, mapping $ f \in \BC(X, Y) $ to $ f^* \in \BC(Y, X) $. In particular we have $ 1^* = 1 $ for all identity morphisms. 

\begin{definition}
A \emph{$ C^* $-category} is a $ * $-category $ \BC $ such that all morphism spaces $ \BC(V,W) $ are Banach spaces, the composition 
maps $ \BC(X,Y) \times \BC(Y, Z), (f,g) \mapsto g \circ f $ satisfy $ \|g \circ f\| \leq \|g\| \|f \| $, the $ C^* $-identity 
$$
||f^* \circ f || = ||f ||^2 
$$ 
holds, and $ f^* \circ f \in \BC(X,X) $ is positive for all $ f \in \BC(X, Y) $. 
\end{definition} 

We note that the last condition means that for any $ f \in \BC(X,Y) $ there exists $ g \in \BC(X,X) $ such that $ f^* \circ f = g^* \circ g $. This is a technical requirement needed to exclude certain pathological phenomena. 
It is automatically satisfied for \emph{additive $ C^* $-categories} in the sense discussed further below. 

A basic example of a $ C^* $-category is the category $ \HILB $ of all Hilbert spaces with morphism spaces $ \HILB(\Hc,\K) = B(\Hc,\K) $ given by all bounded linear operators 
between them. Inside $ \HILB $ we have the $ C^* $-category $ \Hilb $ of all finite dimensional Hilbert spaces. Another basic example is to take a unital $ C^* $-algebra $ A $ and view it as a $ C^* $-category $ \BC_A $ with a single object $ \star $ and 
morphism space $ \BC_A(\star, \star) = A $. From this perspective, $ C^* $-categories can be viewed as a ``many-object'' generalisation of (unital) $ C^* $-algebras. 

\begin{definition}
A linear $ * $-functor between $ C^* $-categories $ \BC, \BD $ is a functor $ \Bf: \BC \rightarrow \BD $ such that all associated maps on morphism spaces are linear and $ \Bf(f^*) = \Bf(f)^* $ for all morphisms $ f $. 
\end{definition} 

If $ \Bf: \BC \rightarrow \BD $ is a linear $ * $-functor between $ C^* $-categories then the maps $ \Bf: \BC(X, Y) \rightarrow \BD(\Bf(X), \Bf(Y)) $ are automatically contractive. The proof of this fact is similar to the argument showing that $ * $-homomorphisms between $ C^* $-algebras are contractive.

A \emph{unitary natural isomorphism} $ \tau: \Bf \Rightarrow \Bg $ between linear $ * $-functors $  \Bf, \Bg: \BC \rightarrow \BD $ is a natural transformation such that $ \tau(X) $ is a unitary isomorphism for all $ X \in \BC $. 
Here a unitary isomorphism is a morphism $ u $ such that $ u^* \circ u = 1 $ and $ u \circ u^* = 1 $. 
We write $ \Bf \cong \Bg $ if there exists a unitary natural isomorphism from $ \Bf $ to $ \Bg $. Two $ C^* $-categories $ \BC, \BD $ are called \emph{equivalent} if there exist linear $ * $-functors $ \Bf: \BC \rightarrow \BD $ 
and $ \Bg: \BD \rightarrow \BC $ such that $ \Bf \circ \Bg \cong \id $ and $ \Bg \circ \Bf \cong \id $. 

A linear $ * $-functor $ \Bf: \BC \rightarrow \BD $ is called \emph{faithful} if the associated maps $ \Bf: \BC(X,Y) \rightarrow \BD(\Bf(X), \Bf(Y)) $ are 
injective for all $ X, Y \in \BC $. Every (small) $ C^* $-category admits a faithful embedding into $ \HILB $, compare \cite[Theorem 6.12]{MITCHENER_cstarcategories}.  

\begin{definition} \label{defdirectsum}
Let $ \BC $ be a $ C^* $-category and let $ X_1, \dots, X_n $ be objects in $ \BC $. Then a \emph{direct sum} of $ X_1, \dots, X_n $ 
is an object $ X = X_1 \oplus \cdots \oplus X_n \in \BC $ together with morphisms $ \iota_j: X_j \rightarrow X $ 
such that $ \iota_k^* \circ \iota_l = \delta_{k,l} 1 $ for all $ 1 \leq k,l \leq n $ and 
$$
\sum_{j = 1}^n \iota_j \circ \iota_j^* = 1.  
$$ 
We say that $ \BC $ is additive if every finite family of objects in $ \BC $ has a direct sum. 
\end{definition} 

By definition, an additive $ C^* $-category $ \BC $ is required to contain a zero object, corresponding to the direct sum of the empty family of objects in Definition \ref{defdirectsum}. That is, there is an object $ 0 \in \BC $ such that $ \BC(0, X) = 0 = \BC(X,0) $ for all $ X \in \BC $. 
The $ C^* $-category $ \HILB $ of Hilbert spaces is additive, with direct sums of Hilbert spaces in the usual sense. 

A nonzero object $ X $ in an additive $ C^* $-category $\BC $ is called \emph{irreducible}, or \emph{simple}, if $ \BC(X,X) = \mathbb{C} \id $. Moreover $ \BC $ is called \emph{semisimple} if 
every object is isomorphic to a direct sum of simple objects. The category $ \Hilb $  of finite dimensional Hilbert spaces is an example of a semisimple $ C^* $-category. 

Let $ \BC $ be a $ C^* $-category and $ X \in \BC $. A \emph{subobject} of $ X $ is an object $ U \in \BC $ together with a 
morphism $ \iota: U \rightarrow X $ such that $ \iota^* \circ \iota = 1 $. In this case $ p = \iota \circ \iota^* \in \BC(X,X) $ is a projection.
Conversely, we say that a projection $ p \in \BC(X,X) $ is \emph{split} if there exists a subobject $ U $ of $ X $ with $ \iota: U \rightarrow X $ 
such that $ p = \iota \circ \iota^* $. The $ C^* $-category category $ \BC $ is called \emph{subobject complete} if every projection in $ \BC $ is split. 

The $ C^* $-categories $ \HILB $ and $ \Hilb $ are subobject complete, whereas $ \BC_A $ for a unital $ C^* $-algebra $ A $ is typically not. We will only be interested in subobject complete $ C^* $-categories. 

\subsubsection{$ C^* $-tensor categories} 

The notion of an additive $ C^* $-category captures already much of the structure present in the category of quantum permutations. More precisely, the quantum permutations of a finite set $ X $ form an additive $ C^* $-category $ \PERM^+(X) $ with morphisms given by intertwiners and direct sums as defined in the previous section. What is still missing from the picture is the tensor product operation. This leads us to the following definition. 

\begin{definition} \label{deflinearcstartensor}
A \emph{$ C^* $-tensor category} is a subobject complete additive $ C^* $-category $ \BC $ together with 
\begin{bnum}
\item[$ \bullet $] a bilinear $ * $-functor $ \otimes: \BC \times \BC \rightarrow \BC $, 
\item[$ \bullet $] a simple object $ 1 \in \BC $, 
\item[$ \bullet $] a unitary natural isomorphism 
$$ 
\alpha: \otimes \circ (\otimes \times \id) \Rightarrow \otimes \circ (\id \times \otimes) 
$$ 
written 
$$ 
\alpha_{X, Y, Z}: (X \otimes Y) \otimes Z \rightarrow X \otimes (Y \otimes Z)
$$ 
for all $ X, Y, Z \in \BC $, and called associator, 
\item[$ \bullet $] unitary natural isomorphisms $ \lambda: 1 \otimes - \Rightarrow \id $ and $ \rho: - \otimes 1 \Rightarrow \id $, called left and right unitors, written $ \lambda_X: 1 \otimes X \rightarrow X $ and $ \rho_X: X \otimes 1 \rightarrow X $ for $ X \in \BC $, respectively. 
\end{bnum}
These data are supposed to satisfy the following conditions. 
\begin{bnum}
\item[$\bullet$] (Associativity constraints) For all objects $ W, X, Y, Z \in \BC $ the diagram
\begin{center}
\begin{tikzcd}[column sep = -12mm]
			&&
			(W \otimes X) \otimes (Y \otimes Z)
			\arrow[drr,"\alpha_{W, X, Y \otimes Z}"]
			&&
			\\
			((W \otimes X) \otimes Y) \otimes Z
			\arrow[urr,"\alpha_{W \otimes X, Y, Z}"]
			\arrow[dr, swap, "\alpha_{W, X, Y} \otimes \id"]
			&&&&
			W \otimes (X \otimes (Y \otimes Z))
			\\
			&
			(W \otimes (X \otimes Y)) \otimes Z
			\arrow[rr, swap, "\alpha_{W, X \otimes Y, Z}"]
			&&
			W \otimes ((X \otimes Y) \otimes Z)
			\arrow[ur, swap, "\id \otimes \alpha_{X, Y, Z}"]
			&
\end{tikzcd}
\end{center}
is commutative. 
\item[$\bullet$] (Unit constraints) $ \lambda_1 = \rho_1 $ and for all objects $ X, Y \in \BC $ the diagram 
\begin{center}
\begin{tikzcd}
			(X \otimes 1) \otimes Y
			\arrow[rr, "\alpha_{X, 1, Y}"]
			\arrow[dr, swap, "\rho_X \otimes \id"]
			&&
			X \otimes (1 \otimes Y)
			\arrow[dl, "\id \otimes \lambda_Y"]
			\\
			&
			X \otimes Y
\end{tikzcd}
\end{center}
is commutative.  
\end{bnum} 
The category $ \BC $ is called strict if $ \alpha, \lambda $ and $ \rho $ are equal to the identity. 
\end{definition} 

If $ \BC, \BD $ are $ C^* $-tensor categories then a unitary tensor functor from $ \BC $ to $ \BD $ is a linear $ * $-functor $ \Bf: \BC \rightarrow \BD $ together with a unitary natural isomorphism $ \Bf_2: \Bf(X) \otimes \Bf(Y) \rightarrow \Bf(X \otimes Y) $ and a unitary 
isomorphism $ \Bf_0: 1_\BD \rightarrow \Bf(1_\BC) $ satisfying some compatibility conditions, 
see \cite[Definition 2.1.3]{NESHVEYEV_TUSET_compactquantumgroups}. 
A unitary natural transformation $ \tau: \Bf \Rightarrow \Bg $ 
between unitary tensor functors is called \emph{monoidal} if it is compatible 
with $ \Bf_2, \Bg_2 $ and $ \Bf_0, \Bg_0 $ as in \cite[Definition 2.1.5]{NESHVEYEV_TUSET_compactquantumgroups}. We write $ \Bf \simeq \Bg $ in this case. Finally, we say that a unitary tensor functor $ \Bf: \BC \rightarrow \BD $ is a \emph{monoidal equivalence} if there exists a unitary tensor functor $ \Bg: \BD \rightarrow \BC $ and monoidal unitary natural isomorphisms $ \Bg \circ \Bf \simeq \id_\BC$ and $ \Bf \circ \Bg \simeq \id_\BD $. 

We will be mostly interested in $ C^* $-tensor categories which are \emph{rigid}. Rigidity means that every object $ X \in \BC $ admits a conjugate object $ \overline{X} $ together with morphisms $ X \otimes \overline{X} \rightarrow 1 $ and $ \overline{X} \otimes X \rightarrow 1 $ satisfying the so-called zig-zag equations, compare 
\cite[Definition 2.2.1]{NESHVEYEV_TUSET_compactquantumgroups}.

The category $ \HILB $ of Hilbert spaces is a $ C^* $-tensor category with the Hilbert space tensor product and tensor unit $ \mathbb{C} $. 
The associativity constraints are the canonical isomorphisms $ (X \otimes Y) \otimes Z \cong X \otimes (Y \otimes Z), (x \otimes y) \otimes z \mapsto x \otimes (y \otimes z) $ 
in this case, and similarly for the unit constraints. The category $ \Hilb $ of finite dimensional Hilbert spaces is rigid with 
the conjugate objects given by the conjugate Hilbert spaces. Here we note that the conjugate Hilbert space makes sense for an arbitrary Hilbert space, but the zig-zag equations only hold in the finite dimensional case. 

For our purposes, the most important example of a $ C^* $-tensor category which we have seen so far is the category $ \PERM^+(X) $ of quantum permutations of a finite set $ X $. Of course, the tensor product operation of $ \PERM^+(X) $ is given by the tensor product of quantum permutations, see Definition \ref{def_constructions}. The associativity and unit constraints are the ``obvious'' ones inherited from $ \HILB $. In the same way we obtain the rigid $ C^* $-tensor category $ \Perm^+(X) $ of all \emph{finite dimensional} quantum permutations of $ X $ as a full subcategory of $ \PERM^+(X) $.

\subsection{Discrete quantum groups} \label{sec_dqg}

A discrete group $ G $ can be described by the $ C^* $-algebra $ C_0(G) $ 
together with the nondegenerate $ * $-homomorphism $ \Delta: C_0(G) \rightarrow M(C_0(G) \otimes C_0(G)) = C_b(G \times G) $ given by $ \Delta(f)(s,t) = f(st) $ for all $ s,t \in G $. Here we recall that the $ C^* $-algebra $ C_b(G \times G) $ of all bounded functions on $ G \times G $ is the multiplier algebra of $ C_0(G) \otimes C_0(G) $, see \cite[Example 3.1.3]{MURPHY_book}.  

The associativity of the group law is equivalent to the \emph{coassociativity} condition $ (\Delta \otimes \id) \Delta = (\id \otimes \Delta) \Delta $ for the \emph{comultiplication} $ \Delta $. Here we tacitly use that nondegenerate $ * $-homomorphisms extend naturally to multiplier algebras. 
Moreover, one checks that $ \Delta(C_0(G))(1 \otimes C_0(G)) $ is a $ * $-subalgebra of $ C_0(G) \otimes C_0(G) $ which separates the points strictly, and hence is 
dense by the Stone-Weierstra{\ss} theorem. Let us recall that if $ X $ is a locally compact Hausdorff space then a $ * $-algebra $ \mathcal{A} \subset C_0(X) $ is said to separate the points of $ X $ strictly if for all $ x, y \in X $ with $ x \neq y $ there exists $ f \in \mathcal{A} $ such that $ f(x) \neq f(y) $, and for all $ x \in X $ there exists $ f \in \mathcal{A} $ with $ f(x) \neq 0 $. 
In the same way one finds that $ \Delta(C_0(G))(C_0(G) \otimes 1) \subset C_0(G) \otimes C_0(G) $ is dense. 

Next consider the nondegenerate $ * $-homomorphism $ \epsilon: C_0(G) \rightarrow \mathbb{C} $ given by evaluation at the identity element, 
so that $ \epsilon(f) = f(e) $. The fact that $ e $ is the identity of $ G $ is encoded in the \emph{counit} relation $ (\epsilon \otimes \id)\Delta = \id = (\id \otimes \epsilon)\Delta $. 

Finally, we have the \emph{Haar measure} on $ G $, which is given by counting measure. This corresponds to the \emph{weight} $ \varphi: C_0(G)_+ \rightarrow [0, \infty] $ 
given by  
$$
\varphi(f) = \sum_{s \in G} f(s). 
$$
Here we write $ C_0(G)_+ $ for the set of all positive functions $ f \geq 0 $ in $ C_0(G) $. 
Note that $ \varphi $ is \emph{faithful}, that is, $ \varphi(f) = 0 $ for $ f \geq 0 $ if and only if $ f = 0 $, and \emph{densely defined}, in the sense that the set of all $ f \in C_0(G)_+ $ with $ \varphi(f) < \infty $ is dense in $ C_0(G)_+ $. 

Discrete \emph{quantum} groups generalise the above structures and properties, by allowing us to replace the points of $ G $ by finite dimensional matrix algebras. 
In the sequel we write $ [X] $ for the closed linear span of a subset $ X $ of a Banach space.  

\begin{definition} \label{defdiscreteqg}
A discrete quantum group is given by a $ C^* $-algebra $ H $ together with a nondegenerate $ * $-homomorphism $ \Delta: H \rightarrow M(H \otimes H) $ such that 
the following conditions are satisfied.  
\begin{bnum} 
\item[a)] (Discreteness) The $ C^* $-algebra $ H = \bigoplus_{i \in I} M_{n_i}(\mathbb{C}) $ is a $ C^* $-direct sum of matrix algebras. 
\item[b)] (Coassociativity) We have $ (\Delta \otimes \id)\Delta = (\id \otimes \Delta)\Delta $. 
\item[c)] (Density conditions) We have $ [\Delta(H)(1 \otimes H)] = H \otimes H = [(H \otimes 1)\Delta(H)] $.
\item[d)] (Counit) There exists a nondegenerate $ * $-homomorphism $ \epsilon: H \rightarrow \mathbb{C} $ such that 
$$ 
(\epsilon \otimes \id)\Delta = \id = (\id \otimes \epsilon)\Delta.
$$
\item[e)] (Invariant weights) There exist densely defined faithful weights $ \varphi, \psi $ on $ H $ such that 
$$ 
(\id \otimes \varphi)\Delta(f) = \varphi(f) 1, \qquad (\psi \otimes \id)\Delta(f) = \psi(f) 1 
$$ 
for all positive elements $ f \in H $. 
\end{bnum} 
\end{definition} 

The last condition in this definition requires some explanation. By 
a \emph{densely defined faithful weight} on a $ C^* $-algebra of the form $ H = \bigoplus_{i \in I} M_{n_i}(\mathbb{C}) $ 
we mean 
a map $ \omega: H_+ \to [0, \infty] $ of the form 
$$
\omega(f) = \sum_{i \in I} \tr(Q_i^{1/2} f_i Q_i^{1/2}) 
$$
for some invertible positive matrices $ Q_i \in M_{n_i}(\mathbb{C}) $, where $ f = (f_i)_{i \in I} $ is a positive element in $ H $. Note that the set of positive elements $ f $ with $ \omega(f) < \infty $ is dense in $ H_+ $, since it contains all $ f = (f_i)_{i \in I} \in H_+ $ with only finitely many nonzero components $ f_i $, and that $ \omega(f) = 0 $ for $ f \in H_+ $ iff $ f = 0 $. These two properties correspond to $ \omega $ being \emph{densely defined} and \emph{faithful}, respectively. Such a weight extends to positive elements in the multiplier algebra in a canonical way, and the invariance conditions in $ e) $ should be interpreted as requiring 
$$
\varphi((\theta \otimes \id)\Delta(f)) = \varphi(f) \theta(1), \qquad \psi((\id \otimes \theta)\Delta(f)) = \psi(f) \theta(1) 
$$
for all positive linear functionals $ \theta $ on $ H $ and $ f \in H_+ $. 

We also note that we always work with the \emph{minimal tensor product} of $ C^* $-algebras. In the case of a discrete quantum group, the underlying $ C^* $-algebra $ H $ is automatically nuclear by condition $ a) $ in Definition \ref{defdiscreteqg}, so that there is actually no ambiguity in the choice of the tensor product.

In analogy with the case of discrete groups, we will often formally write $ H = C_0(\GG) $ and say that $ \GG $ \emph{is a discrete quantum group}. 
In fact, if all matrix blocks in $ C_0(\GG) $ are one-dimensional then one can show that there is a classical discrete group $ G $ 
such that $ C_0(\GG) = C_0(G) $ is the algebra of functions on $ G $ vanishing at infinity, with the structure maps explained at the beginning. In this sense the theory of discrete quantum groups generalises discrete groups. 

Note that the underlying algebra structure of $ C_0(\GG) $ for a discrete quantum group $ \GG $ is not particularly interesting, and the key structural properties are encoded in the coproduct. The same is true classically: namely, the underlying set of a discrete group does not carry much information. 

From a discrete quantum group $ \GG $ one obtains a $ C^* $-tensor category $ \BC_\GG $ as follows. The objects of $ \BC_\GG $ are all finite dimensional nondegenerate $ * $-representations $ \pi: C_0(\GG) \rightarrow B(\Hc_\pi) $, and morphisms are the intertwiners between them. 
The tensor product of $ \pi, \rho \in \BC_\GG $ is the tensor product $ \Hc_\pi \otimes \Hc_\rho $ of the underlying Hilbert spaces, together with $ (\pi \otimes \rho) \Delta $. That is, we compose the comultiplication with the extension to multipliers of the map 
$$ 
\pi \otimes \rho: C_0(\GG) \otimes C_0(\GG) \rightarrow B(\Hc_\pi) \otimes B(\Hc_\rho) = B(\Hc_\pi \otimes \Hc_\rho).  
$$ 
The tensor unit is $ \mathbb{C} $, viewed as a nondegenerate $ * $-representation via the counit $ \epsilon $. 
It can be shown that the category $ \BC_\GG $ is always rigid.

\subsection{Tannaka-Krein duality} 

The construction at the very end of the previous section provides 
a link between discrete quantum groups and rigid $ C^* $-tensor categories. In fact, it turns out that every rigid $ C^* $-tensor category which admits a faithful unitary tensor functor 
to the category $ \Hilb $ of finite dimensional Hilbert spaces is of the form $ \BC_\GG $ for a discrete quantum group $ \GG $. 
This is the content of the celebrated \emph{Tannaka-Krein duality} theorem which we discuss next. 

Let $ \BC $ be a rigid $ C^* $-tensor category. By a \emph{fibre functor} we mean a unitary tensor functor $ \Bf: \BC \rightarrow \Hilb $ which is faithful, that is, injective 
on morphisms. Given a discrete quantum group $ \GG $, the forgetful functor $ \iota: \BC_\GG \rightarrow \Hilb $, which assigns to a finite dimensional 
nondegenerate $ * $-representation $ \pi: C_0(\GG) \rightarrow B(\Hc_\pi) $ the underlying Hilbert space $ \Hc_\pi $, is an example of a fibre functor. 

\begin{theorem}[Tannaka-Krein] \label{TK}
Let $ \BC $ be a rigid $ C^* $-tensor category together with a fibre functor $ \Bf: \BC \rightarrow \Hilb $. Then there exists a discrete quantum group $ \GG $ 
and a monoidal equivalence $ F: \BC \rightarrow \BC_\GG $ such that $ \Bf \simeq \iota \circ F $. 
\end{theorem} 

We shall not give a proof of Theorem \ref{TK} here and refer the reader to \cite[Section 2.3]{NESHVEYEV_TUSET_compactquantumgroups}. Let us remark, however, that 
in \cite{NESHVEYEV_TUSET_compactquantumgroups} this theorem is presented in terms of \emph{compact} quantum groups instead of discrete ones.  
This is equivalent to our formulation above, as we will briefly explain further below. 

We have chosen the language of discrete quantum groups instead of compact ones because this is more natural from the point of view of the examples we are interested in. In fact, using Theorem \ref{TK} we can immediately give the following definition.  

\begin{definition} \label{defqpermgroupfinite}
Let $ X $ be a finite set. The quantum permutation group $ \Sym^+(X) $ is the discrete quantum group associated via Tannaka-Krein reconstruction to the rigid $ C^* $-tensor category $ \Perm^+(X) $ of finite dimensional quantum permutations of $ X $ and its natural forgetful functor to $ \Hilb $. 
In the case $ X = [n] = \{1, \dots, n\} $ we also write $ \Sym^+_n $ for this quantum group. 
\end{definition}

In section \ref{sec_infinite} we will provide a more concrete description of $ \Sym^+(X) $, including formulas for the $ C^* $-algebra $ C_0(\Sym^+(X)) $ and its coproduct, in slightly greater generality.

Let us now come back to compact quantum groups and their role in Tannaka-Krein reconstruction. Compact quantum groups are actually slightly easier to define in comparison with Definition \ref{defdiscreteqg}. 

\begin{definition} \label{defcompactqg}
A compact quantum group is given by a $ C^* $-algebra $ H $ together with a nondegenerate $ * $-homomorphism $ \Delta: H \rightarrow M(H \otimes H) $ such that 
the following conditions are satisfied.  
\begin{bnum} 
\item[a)] (Compactness) The $ C^* $-algebra $ H $ is unital. 
\item[b)] (Coassociativity) We have $ (\Delta \otimes \id)\Delta = (\id \otimes \Delta)\Delta $. 
\item[c)] (Density conditions) We have $ [\Delta(H)(1 \otimes H)] = H \otimes H = [(H \otimes 1)\Delta(H)] $. 
\end{bnum} 
\end{definition} 

By condition $ a) $, the comultiplication $ \Delta $ in a compact quantum group $ H $ is a unital $ * $-homomorphism from $ H $ to $ H \otimes H $, and there is no need to consider multipliers. We have started with $ \Delta $ as a map from $ H $ to $ M(H \otimes H) $ only in order to emphasise the similarity with Definition \ref{defdiscreteqg}.  

Note that in Definition \ref{defcompactqg} there is no mention of a counit or invariant weights. It turns out that every compact quantum group automatically admits a uniquely determined left and right invariant state, called the \emph{Haar state}. There is also a counit, which however only exists as a densely defined unbounded operator in general. We refer the reader to \cite{WORONOWICZ_compactquantumgroups} for more information. 

The prototpyical example of a compact quantum group is the $ C^* $-algebra $ H = C(G) $ of functions on a compact group $ G $, with the coproduct induced by 
the group multiplication. In this case the counit, given by evaluation at the identity element, makes perfect sense as a $ * $-homomorphism from $ H $ to $ \mathbb{C} $. 

Crucially, there is a duality between compact and discrete quantum groups, generalising \emph{Pontrjagin duality} between compact and discrete abelian groups. 
That is, compact and discrete quantum groups are essentially equivalent ways of describing the same mathematical structure. This is the reason why our statement of Theorem \ref{TK} is equivalent to the formulation in \cite{NESHVEYEV_TUSET_compactquantumgroups}.

\subsection{Discretisation} 

There is another link between compact and discrete quantum groups which is relevant to our discussion. 
If $ G $ is a compact group then the underlying set of $ G $, which we shall denote by $ G_\delta $, can be viewed as a discrete group. 
Note that this discrete group will typically be ``large''; for instance, the discretisation of the circle group $ U(1) $ is uncountable. 

On the level of the function algebra $ C(G) $, we can view the passage from a compact group $ G $ to its discretisation $ G_\delta $ as passing from $ G $ to the $ C^* $-tensor category $ \Rep(C(G)) $ 
of all finite dimensional unital $ * $-representations of $ C(G) $, with the tensor product operation defined using the coproduct $ \Delta: C(G) \rightarrow C(G \times G) $ in complete analogy to the discussion at the end of section \ref{sec_dqg}.
In fact, the irreducible $ * $-representations of $ C(G) $ are all one-dimensional, and correspond precisely to point evaluations at the elements of $ G $. Moreover, the tensor product of point evaluations is given by the group law, the unit object is given by evaluation at the identity, and conjugates are obtained by taking inverses in $ G $. The $ C^* $-tensor category $ \Rep(C(G)) $ is therefore nothing but the category $ \BC_{G_\delta} $ associated to the discretisation $ G_\delta $ of $ G $. 

If $ G $ is compact and abelian then discretisation of $ G $ is closely related to \emph{Bohr compactification} of its Pontrjagin dual. More precisely, if $ \widehat{G} $ denotes the Pontrjagin dual of $ G $, then the Pontrjagin dual of the discretisation $ G_\delta $ identifies with  
the Bohr compactification of $ \widehat{G} $. That is, discretisation and Bohr compactification correspond to each other under Pontrjagin duality. 

It turns out that the concepts of discretisation and Bohr compactification carry over naturally to the realm of compact and discrete quantum groups, respectively \cite{SOLTAN_bohr}. 
For our purposes, the key point is that the discrete quantum groups we have encountered so far are in fact discretisations of compact quantum groups. More specifically, the discrete quantum group $ \Sym^+_n $ from Definition \ref{defqpermgroupfinite} is the discretisation of the compact quantum group $ S_n^+ $ introduced by Wang \cite{WANG_quantumsymmetry}. We shall not spell out the definition of $ S_n^+ $ here, suffice it to say that $ C(S_n^+) $ is the universal unital $ C^* $-algebra generated by a magic unitary $ n \times n $-matrix. 

The quantum permutation groups $ S_n^+ $ and their quantum subgroups have been studied intensively, and they have a rich combinatorial structure through their representation theory \cite{BANICA_quantumpermutationgroups}, \cite{FRESLON_compactquantumgroupscombinatorics}. 
However, for most of the constructions which we shall 
discuss in the remainder of these notes there is no companion in the world of compact quantum groups, and one really needs to work in the setting of discrete quantum groups from the outset. 

\subsection{Exercises}

\begin{exercise} \label{ex21}
Show that every finite dimensional quantum permutation of a finite set decomposes into a direct sum of irreducible quantum permutations. 
\end{exercise}

\begin{exercise} \label{ex22}
Show that every (possibly infinite dimensional) classical quantum permutation of a finite set is equivalent to a direct sum of classical permutations. 
\end{exercise}

\begin{exercise}
Let $ G $ be a monoid (equipped with the discrete topology). Moreover let $ H = C_0(G) $ and $ \Delta: C_0(G) \rightarrow M(C_0(G) \otimes C_0(G)) = C_b(G \times G) $ be given by $ \Delta(f)(s,t) = f(st) $. Show that the density conditions 
$$ 
[\Delta(H)(1 \otimes H)] = H \otimes H = [(H \otimes 1)\Delta(H)] $$
are equivalent to $ G $ having left and right cancellation. That is, for any $ s,r,t \in G $ the equality $ ts = tr $ implies $ s = r $, and similarly $ st = rt $ implies $ s = r $. 
\end{exercise}

\begin{exercise}
Show that a compact topological monoid with left and right cancellation is a group. Conclude that a compact quantum group $ H $ whose underlying $ C^* $-algebra is abelian is of the form $ H \cong C(G) $ for a compact group $ G $. 
\end{exercise}

\begin{exercise}
Let $ H $ be a commutative $ C^* $-algebra together with a nondegenerate $ * $-homomorphism $ \Delta: H \rightarrow M(H \otimes H) $ such that the first four conditions in Definition \ref{defdiscreteqg} are satisfied. Is $ H $ necessarily a discrete quantum group? 
\end{exercise}

\section{Infinite quantum symmetries} \label{sec_infinite}

In our discussion so far we have restricted our attention to finite sets. We shall now  lift the finiteness assumptions, and consider quantum permutations of arbitrary sets, following \cite{VOIGT_infinitequantumpermutations}. This will allow us to study quantum symmetries of infinite graphs, generalising the constructions in section \ref{sec_qutgraph}.

\subsection{Infinite quantum permutations} 

In order to extend Definition \ref{defqperm} to the case when $ X $ is an infinite set, the second part of the conditions needs attention. 
Recall that if $ \Hc $ is a Hilbert space then the \emph{strong operator topology} on $ B(\Hc) $ is determined by saying that a net $ (T_i)_{i \in I} $ converges to $ T $ strongly 
if $ T_i(\xi) \rightarrow T(\xi) $ in norm for all $ \xi \in \Hc $. 

\begin{definition} \label{defquantumpermutation}
Let $ X $ be a set. A quantum permutation of $ X $ is a pair $ (\Hc, u) $ consisting of a Hilbert space $ \Hc $ 
and a family $ u = (u_{x,y})_{x,y \in X} $ of elements $ u_{x,y} \in B(\Hc) $ such that 
\begin{bnum} 
\item[$\bullet$] For every $ x, y \in X $ we have $ u_{x,y}^2 = u_{x,y} = u_{x,y}^* $. 
\item[$\bullet$] We have 
$$ 
\sum_{t \in X} u_{x,t} = 1 = \sum_{t \in X} u_{t,y} 
$$
for all $ x,y \in X $, with convergence understood in the strong operator topology. 
\end{bnum}
We also say that $ u $ is a magic unitary in this case. 
If $ \sigma = (\Hc_\sigma, u^\sigma) $ and $ \tau = (\Hc_\tau, u^\tau) $ are quantum permutations of $ X $ then an intertwiner from $ \sigma $ to $ \tau $ is a 
bounded linear operator $ T: \Hc_\sigma \rightarrow \Hc_\tau $ such that $ T u^\sigma_{x,y} = u^\tau_{x,y} T $ for all $ x, y \in X $. 
\end{definition} 

It is not hard to check that the convergence of the infinite sums in Definition \ref{defquantumpermutation} can be interpreted equivalently in any of 
the weak, strong, strong*, $ \sigma $-weak, $ \sigma $-strong or $ \sigma $-strong* topologies. In contrast, asking for convergence in the norm topology is a too strong requirement if $ X $ is infinite. 

\begin{lemma} \label{orthogonality}
The projections in each row and each column of a magic unitary are pairwise orthogonal. 
\end{lemma}  

\begin{proof} 
Fix a set $ X $. It suffices to show that if $ \sum_{x \in X} p_x = 1 $ is a strongly convergent sum of projections in $ B(\Hc) $ for some Hilbert space $ \Hc $ then $ p_x p_y = 0 $ for $ x \neq y $. 
To this end let $ \xi \in \Hc $ and $ y \in X $, and note that $ p_y \xi = p_y 1 p_y \xi = \sum_{x \in X} p_y p_x p_y \xi $ since multiplication is separately strongly continuous. 
It follows that $ \sum_{x \neq y} p_y p_x p_y \xi = 0 $, and hence  
$$
0 = \bra \xi, \sum_{x \neq y} p_y p_x p_y \xi \ket = \sum_{x \neq y} \bra p_x p_y \xi, p_x p_y \xi \ket. 
$$ 
All terms on the right hand side of this equality are nonnegative, and therefore $ p_x p_y \xi = 0 $ for all $ x \neq y $. Since $ \xi $ was arbitrary 
it follows that $ p_x p_y = 0 $ as required.
\end{proof} 

In the sequel, all infinite sums of families $ (p_i)_{i \in I} $ of pairwise orthogonal projections in a Hilbert space will be understood in the strong operator topology, 
and if $ \sum_{i \in I} p_i = 1 $ then we call $ (p_i)_{i \in I} $ a \emph{partition of unity}. We can rephrase Definition \ref{defquantumpermutation} 
by saying that a quantum permutation of a set $ X $ is a matrix of projections, indexed by $ X $, such that all rows and columns form partitions of unity. 

One defines direct sums, tensor products, and conjugates of quantum permutations of infinite sets in the same way as in the finite case, see Definition \ref{def_constructions}. 
Moreover, using the same arguments as in the proof of Lemma \ref{tensorstructure} one checks that these operations yield again quantum permutations. 

\begin{definition} 
Let $ X $ be a set. We denote by $ \PERM^+(X) $ the $ C^* $-tensor category of all quantum permutations of $ X $ and their intertwiners. Similarly, we write $ \Perm^+(X) $ for the $ C^* $-tensor category of all finite dimensional quantum permutations. 
\end{definition} 

In the same way as for finite sets, the $ C^* $-tensor category $ \Perm^+(X) $ is rigid. The Tannaka-Krein reconstruction theorem \ref{TK} 
therefore allows us to give the following definition. 

\begin{definition}
Let $ X $ be a set. The quantum permutation group $ \Sym^+(X) $ is the discrete quantum group associated to the rigid $ C^* $-tensor category 
of $ \Perm^+(X) $ of finite dimensional quantum permutations of $ X $ and the canonical forgetful functor to $ \Hilb $ via Tannaka-Krein reconstruction. 
\end{definition}

Concretely, the $ C^* $-algebra of functions on $ \Sym^+(X) $ can be written as the $ C^* $-direct sum of matrix algebras
$$
C_0(\Sym^+(X)) = \bigoplus_{\sigma} B(\Hc_\sigma), 
$$
taken over the set of isomorphism classes of finite dimensional irreducible quantum permutations $ \sigma = (\Hc_\sigma, u^\sigma) $ of $ X $. 

In order to describe the quantum group structure of $ C_0(\Sym^+(X)) $, it is convenient to work with the \emph{universal quantum permutation} of $ X $, by which 
we mean the family $ u = (u_{x,y})_{x,y \in X} $ of elements $ u_{x,y} \in M(C_0(\Sym^+(X))) $ with components $ u^\sigma_{x,y} $. 
The coproduct of $ C_0(\Sym^+(X)) $ is the uniquely determined nondegenerate $ * $-homomorphism $ \Delta: C_0(\Sym^+(X)) \rightarrow M(C_0(\Sym^+(X)) \otimes C_0(\Sym^+(X))) $ satisfying 
$$
\Delta(u_{x,y}) = \sum_{t \in X} u_{x,t} \otimes u_{t,y} 
$$
for all $ x,y \in X $. 
The counit $ \epsilon: C_0(\Sym^+(X)) \rightarrow \mathbb{C} $ is the unique (nondegenerate) $ * $-homomorphism satisfying $ \epsilon(u_{x,y}) = \delta_{x,y} $ for $ x,y \in X $. 
Moreover, one can show that a left and right invariant weight $ \varphi $ for $ C_0(\Sym^+(X)) $ is given by  
$$
\varphi(f) = \sum_{\sigma} \dim(\Hc_\sigma) \tr(f_\sigma), 
$$
where $ f = (f_\sigma) $ is contained in the positive part of $ C_0(\Sym^+(X)) $. Here $ \tr $ denotes the trace on $ B(\Hc_\sigma) $, normalised such that $ \tr(1) = \dim(\Hc_\sigma) $ for the identity matrix $ 1 \in B(\Hc_\sigma) $.

\subsection{Universality} 

We shall now explain how quantum permutations can be viewed as \emph{universal quantum symmetries}. 

\begin{definition} \label{defaction}
An action of a discrete quantum group $ \GG $ on a $ C^* $-algebra $ B $ is an injective 
nondegenerate $ * $-homomorphism $ \beta: B \rightarrow M(C_0(\GG) \otimes B) $ such that 
$$ 
(\Delta \otimes \id) \beta = (\id \otimes \beta) \beta 
$$ 
and $ [\beta(B)(C_0(\GG) \otimes 1)] = C_0(\GG) \otimes B $. 
\end{definition} 

Note here that the density condition $ [\beta(B)(C_0(\GG) \otimes 1)] = C_0(\GG) \otimes B $ encodes two requirements: firstly, the linear span of elements $ \beta(b)(f \otimes 1) $ for $ b \in B, f \in C_0(\GG) $ is contained in $ C_0(\GG) \otimes B $, not just $ M(C_0(\GG) \otimes B) $, and secondly, this linear span is dense in there. 

The most basic example of an action, for any discrete quantum group $ \GG $, is obtained by equipping an arbitrary $ C^* $-algebra $ B $ with the \emph{trivial action}, given by $ \beta(b) = 1 \otimes b $ 
for all $ b \in B $. Another example is the \emph{translation action} of $ \GG $ on $ B = C_0(\GG) $, 
defined by $ \beta = \Delta: C_0(\GG) \rightarrow M(C_0(\GG) \otimes C_0(\GG)) $. It follows directly from Definition \ref{defdiscreteqg} 
that this defines indeed an action of $ \GG $ on $ C_0(\GG) $.  

We are mainly interested in the case $ B = C_0(X) $ 
for a set $ X $, viewed as a discrete topological space. Recall from 
the discussion after Definition \ref{defdiscreteqg} the notion of a densely defined faithful weight on a $ C^* $-algebra of this form. 

\begin{definition} \label{definvariantweight}
Let $ B = C_0(X) $ for a set $ X $ and let $ \beta: B \rightarrow M(C_0(\GG) \otimes B) $ be an action of a discrete quantum group $ \GG $ on $ B $. A densely defined faithful weight $ \omega $ on $ B $ 
is called invariant with respect to $ \beta $ if $ \omega((\theta \otimes \id)\beta(b)) = \omega(b) \theta(1) $ for all $ b \in B_+ $ and all positive linear functionals $ \theta $ on $ C_0(\GG) $. 
\end{definition}

We shall say that a discrete quantum group $ \GG $ is the \emph{discrete quantum symmetry group} of $ (B, \omega) $ if there exists an 
action $ \beta: B \rightarrow M(C_0(\GG) \otimes B) $ such that $ \omega $ is invariant with respect to $ \beta $ and the following universal property is satisfied. 
If $ \HH $ is an arbitrary discrete quantum group together with an action $ \gamma: B \rightarrow M(C_0(\HH) \otimes B) $ 
such that $ \omega $ is invariant with respect to $ \gamma $, then there exists a unique morphism of discrete quantum groups $ \iota: \HH \rightarrow \GG $ such that the diagram 
$$
\xymatrix{
B \ar@{->}[r]^{\!\!\!\!\!\!\!\!\!\!\!\!\!\!\!\!\!\!\!\!\!\! \beta} \ar@{->}[rd]_{\gamma} & M(C_0(\GG) \otimes B) \ar@{->}[d]^{\iota^* \otimes \id} \\ 
& M(C_0(\HH) \otimes B) 
}
$$
is commutative. Here, by definition, a morphism $ \iota: \HH \rightarrow \GG $ of discrete quantum groups is a nondegenerate $ * $-homomorphism $ \iota^*: C_0(\GG) \rightarrow M(C_0(\HH)) $ which is compatible with the coproducts in the sense that $ \Delta \circ \iota^* = (\iota^* \otimes \iota^*) \circ \Delta $. 

It is straightforward to check that the discrete quantum symmetry group of a $ C^* $-algebra $ B $ is uniquely determined up to isomorphism. 
If $ X $ is a set and $ B = C_0(X) $ is equipped with the weight induced by counting measure, we shall call the corresponding discrete quantum symmetry group 
the \emph{universal discrete quantum group acting on $ X $}. The following result shows that this quantum group indeed exists, and provides at the same time an alternative characterisation of $ \Sym^+(X) $, compare \cite[Proposition 4.6]{VOIGT_infinitequantumpermutations}. 

\begin{prop} \label{quantumsymmetryuniversal}
Let $ X $ be a set. The quantum permutation group $ \Sym^+(X) $ is the universal discrete quantum group acting on $ X $. 
\end{prop} 

\begin{proof} 
Recall that the universal quantum permutation of $ X $ is $ u = (u_{x,y})_{x,y \in X} $ with $ u_{x,y} \in M(C_0(\Sym^+(X))) $.
If we write $ e_x \in C_0(X) $ for the characteristic function based at $ x \in X $ then the formula 
$$ 
\beta(e_x) = \sum_{y \in X} u_{x,y} \otimes e_y 
$$ 
determines a nondegenerate $ * $-homomorphism  $ \beta $ from $ C_0(X) $ to $ M(C_0(\Sym^+(X)) \otimes C_0(X)) $, and it is straightforward to check that this yields an action of $ \Sym^+(X) $ on $ B = C_0(X) $. Moreover, the weight $ \omega $ on $ C_0(X) $ corresponding to counting measure is invariant with respect to this action. 

Assume that $ \gamma: B \rightarrow M(C_0(\HH) \otimes B) $ is an action of a discrete quantum group $ \HH $ on $ B $ such that $ \omega $ is invariant with respect to $ \gamma $. Firstly, note that we can 
write $ \gamma(e_x) = \sum_{y \in X} v_{x,y} \otimes e_y $ for uniquely determined elements $ v_{x,y} \in M(C_0(\HH)) $. We claim that the elements $ v_{x,y} $, 
or rather their components in each matrix block inside $ C_0(\HH) $, satisfy the defining properties for a quantum permutation. Indeed, note that 
from $ e_x^2 = e_x = e_x^* $ we get $ v_{x,y}^2 = v_{x,y} = v_{x,y}^* $ for all $ x,y \in X $. 
Since $ \gamma $ is nondegenerate we have $ \sum_{x \in X} v_{x,y} = 1 $ for all $ y \in X $. 
Moreover, we obtain $ \sum_{y \in X} v_{x,y} = 1 $ for all $ x \in X $ since the weight $ \omega $ is invariant with respect to $ \gamma $. 

Since $ C_0(\HH) $ is a direct sum of matrix algebras this allows us to define a nondegenerate $ * $-homomorphism $ \iota^*: C_0(\Sym^+(X)) \rightarrow M(C_0(\HH)) $ 
such that $ \iota^*(u_{x,y}) = v_{x,y} $ for all $ x, y \in X $. It is straightforward to check that $ \iota^* $ is compatible with the comultiplications 
and satisfies $ (\iota^* \otimes \id) \beta = \gamma $. Moreover $ \iota^* $ is uniquely determined by these properties. 
\end{proof}

\subsection{Quantum automorphisms of infinite graphs} 

Let $ X = (V_X, E_X) $ be a graph. Recall that we denote by $ V_X $ and $ E_X $ the set of vertices and edges of $ X $, respectively, and we require that $ E_X $ is a symmetric subset  of $ V_X \times V_X $ which has empty intersection with the diagonal. 

Extending the discussion in section \ref{sec_introduction}, we will now define and study quantum automorphisms in the case that $ X $ is  \emph{infinite}, in the sense that both $ V_X $ and $ E_X \subset V_X \times X_X $ can have arbitrary cardinality. 
As for finite graphs, the \emph{adjacency matrix} of an infinite graph $ X $ is the matrix $ A_X \in M_{V_X}(\{0,1\}) $ determined by 
$$
(A_X)_{x,y} = 1 \Leftrightarrow (x,y) \in E_X
$$
for $ x,y \in V_X $. 

\begin{definition} \label{defquantumautomorphismgraph}
Let $ X = (V_X, E_X) $ be a graph. A quantum automorphism of $ X $ is a quantum permutation $ (\Hc, u) $ of $ V_X $ such that 
$$ 
A_X u = u A_X 
$$ 
as matrices in $ M_{V_X}(B(\Hc)) $, where $ A_X $ is the adjacency matrix of $ X $. 
\end{definition} 

We note that the equality in Definition \ref{defquantumautomorphismgraph} makes indeed sense since the matrix entries of the product of $ u $ and $ A_X $ on either side are strongly convergent sums of pairwise orthogonal projections. 

It is straightforward to check that the collection of all quantum automorphisms of $ X $ is closed under taking direct sums, tensor products and conjugates. 
This means that we obtain naturally a $ C^* $-tensor category in this situation. Moreover, finite dimensional quantum automorphisms give rise to an associated discrete quantum group $ \Qut_\delta(X) $ by the Tannaka-Krein theorem \ref{TK}. 

\begin{definition} \label{defqutinfinitegraph}
Let $ X $ be a graph. We shall call the quantum group $ \Qut_\delta(X) $ the discrete quantum automorphism group of $ X $. 
\end{definition} 

For finite graphs, the definition of a compact quantum automorphism group was given by Banica in \cite{BANICA_quthomogeneousgraphs}, modifying an earlier definition by Bichon \cite{BICHON_qutgraphs}. In contrast, the quantum group $ \Qut_\delta(X) $ in Definition \ref{defqutinfinitegraph} is discrete, regardless of whether $ X $ is finite or not. 
It is not hard to show that if $ X $ is a finite graph then 
$$ 
\Qut_\delta(X) = \Qut(X)_\delta 
$$ 
is equal to the discretisation of Banica's quantum automorphism group $ \Qut(X) $ of $ X $. 
In particular, the quantum groups $ \Qut(X) $ and $ \Qut_\delta(X) $ are \emph{not} isomorphic in general. 

Let us mention that another, very elegant approach to quantum symmetries of infinite graphs has been developed by Rollier and Vaes in \cite{ROLLIER_VAES_quantumautomorphismgroups}. 
Their constructions are directly motivated by the work of Man\v{c}inska-Roberson \cite{MANCINSKA_ROBERSON_quantumisoplanar}. 
Whereas Definition \ref{defqutinfinitegraph} works for arbitrary graphs, the theory in \cite{ROLLIER_VAES_quantumautomorphismgroups} only applies to locally finite connected graphs.
At the same time, in the cases where \cite{ROLLIER_VAES_quantumautomorphismgroups} is applicable, it yields a more refined invariant, from which the quantum automorphism group in the sense of Definition \ref{defqutinfinitegraph} can be obtained via discretisation.

\subsection{Examples} 

Let us consider some examples of graphs and their quantum symmetries, mostly following \cite{VOIGT_infinitequantumpermutations}. 

\subsubsection{Disjoint automorphisms} 

Two automorphisms $ \sigma, \tau \in \Aut(X) $ of a graph $ X $ are called \emph{disjoint} if $ \sigma(x) \neq x \implies \tau(x) = x $, or equivalently, $ \tau(x) \neq x \implies \sigma(x) = x $ 
for all $ x \in V_X $. The existence of a pair of disjoint automorphisms is sufficient for a graph to have quantum symmetry, by the following 
result due to Schmidt \cite[Theorem 2.2]{SCHMIDT_foldedcube}, see also \cite[Proposition 7.7]{VOIGT_infinitequantumpermutations}. 

\begin{prop} \label{disjointautomorphisms}
Let $ X $ be a graph admitting a pair of disjoint automorphisms $ \sigma, \tau \in \Aut(X) $, and assume that $ k \in \mathbb{N} $ does not exceed the order
of neither of these automorphisms. Then $ X $ admits an irreducible quantum automorphism of dimension $ k $. In particular, if $ \Aut(X) $ contains a pair of nontrivial 
disjoint automorphisms then $ X $ has quantum symmetry. 
\end{prop} 

\begin{proof} 
Let us fix $ k \in \mathbb{N} $ such that $ k \leq \min(\ord(\sigma), \ord(\tau)) $, and note that if $ \ord(\sigma) = \infty = \ord(\tau) $ this means that we can 
choose $ k $ arbitrarily. We shall construct an irreducible quantum automorphism $ \rho = (\mathbb{C}^k, u^\rho) $ as follows. 

Let $ G = \mathbb{Z}/k\mathbb{Z} $ and consider the actions of $ C(G) $ and $ C^*(G) $ on $ \mathbb{C}^k = l^2(G) $, induced by pointwise multiplication of functions in $ C(G) $, 
and the regular representation of $ G $, respectively. We shall write $ p_1, \dots p_k $ and $ q_1 \dots, q_k $ for the images in $ M_k(\mathbb{C}) = B(l^2(G)) $ of the minimal 
projections in $ C(G) $ and $ C^*(G) $ under these representations, respectively. 
 
Define $ u^\rho = (u^\rho_{x,y})_{x,y \in V_X} $ by $ u^\rho = \sum_{r = 1}^k u^{\sigma^r} p_r + \sum_{s = 1}^k u^{\tau^s} q_s - u^{\id} 1 $. That 
is, the matrix $ u^\rho_{x,y} \in M_k(\mathbb{C}) $ is given by 
$$ 
u^\rho_{x,y} = \sum_{r = 1}^k \delta_{x,\sigma^r(y)} p_r + \sum_{s = 1}^k \delta_{x,\tau^s(y)} q_s - \delta_{x,y} 1
$$ 
for $ x, y \in V_X $. 
Since $ \sigma $ and $ \tau $ are graph automorphisms we clearly have $ u^\rho A_X = A_X u^\rho $. Let 
\begin{align*}
M_{xy} &= \{1 \leq r \leq k \mid \sigma^r(y) = x \}, \\
N_{xy} &= \{1 \leq s \leq k \mid \tau^s(y) = x \}, 
\end{align*}
and observe that 
$$ 
u^\rho_{x,y} = 
\begin{cases} 
\sum_{r \in M_{xy}} p_r & \text{ if } \sigma(y) \neq y  \\
\sum_{s \in N_{xy}} q_s & \text{ if } \tau(y) \neq y  \\
\delta_{x,y} 1 & \text{ if } \sigma(y) = y = \tau(y). 
\end{cases}
$$
In particular, every $ u^\rho_{x,y} $ for $ x, y \in V_X $ is a projection. In addition, we have 
\begin{align*}
\sum_{x \in V_X} u^\rho_{x,y} &= \sum_{r = 1}^k \sum_{x \in V_X} \delta_{x, \sigma^r(y)} p_r + \sum_{s = 1}^k \sum_{x \in V_X} \delta_{x, \tau^s(y)} q_s - 1
= \sum_{r = 1}^k  p_r + \sum_{s = 1}^k q_s - 1 = 1,  
\end{align*}
and similarly 
\begin{align*}
\sum_{y \in V_X} u^\rho_{x,y} &= \sum_{r = 1}^k \sum_{y \in V_X} \delta_{x,\sigma^r(y)} p_r + \sum_{s = 1}^k \sum_{y \in V_X} \delta_{x,\tau^s(y)} q_s - 1 
= \sum_{r = 1}^k p_r + \sum_{s = 1}^k q_s - 1 = 1 
\end{align*}
as required. It follows that $ \rho = (\mathbb{C}^k, u^\rho) $ is a quantum automorphism of $ X $. 

Let us now check that $ \rho $ is irreducible. If $ \ord(\sigma) = m < \infty $ then 
upon decomposing $ X $ into the orbits of $ \sigma $ we see that there are finitely many elements $ v_1, \dots, v_a \in V_X $, fixed under $ \tau $, 
such that $ (\sigma^t(v_1), \dots, \sigma^t(v_a)) = (v_1, \dots, v_a) $ for $ 0 \leq t < k $ implies $ t = 0 $. 
If $ \ord(\sigma) = \infty $ we either find an infinite orbit, or orbits of arbitrarily large finite size. 
Again this allows us to choose $ v_1, \dots, v_a \in V_X $ such that $ (\sigma^t(v_1), \dots, \sigma^t(v_a)) = (v_1, \dots, v_a) $ for $ 0 \leq t < k $ implies $ t = 0 $.
In fact, we may take $ v_1 = \cdots = v_a $ for a suitably chosen vertex in this case. 
In either case, it follows that $ u^\rho_{\sigma^r(v_1), v_1} \cdots u^\rho_{\sigma^r(v_a), v_a} = p_r $. 

In the same way we find $ w_1, \dots, w_b \in V_X $ such that $ u^\rho_{\tau^r(w_1), w_1} \cdots u^\rho_{\tau^r(w_b), w_b} = q_r $. 
Since the projections $ p_i, q_j $ for $ 1 \leq i,j \leq k $ generate $ M_k(\mathbb{C}) $ this yields the claim.
\end{proof}

\subsubsection{Finite trees} 

Recall that a graph $ X = (V_X, E_X) $ is called a \emph{tree} if for all vertices $ x, y \in V_X $ there exists a unique path connecting $ x, y $. 
Here by a path we mean a finite sequence of vertices $ x = x_0, x_1, \dots, x_n = y $ such that $ (x_{i - 1}, x_i) \in E_X $ for all $ i = 1, \dots, n $. 

Proposition \ref{disjointautomorphisms} can be used to show that generic finite trees have quantum symmetry. More precisely, Junk-Schmidt-Weber have shown
in \cite{JUNK_SCHMIDT_WEBER_trees} that almost all finite trees admit a disjoint pair of ``cherries'' in the following sense. 

\begin{definition} 
Let $ X = (V_X, E_X) $ be a graph. A cherry in $ X $ is a triple $ (x, x_1, x_2) $ of pairwise distinct vertices in $ X $ such that 
\begin{bnum} 
\item[a)] $ (x_i, x) \in E_X $ for $ i = 1,2 $. 
\item[b)] The degree of $ x_i $ is one for $ i = 1,2 $. 
\item[c)] The degree of $ x $ is three. 
\end{bnum} 
\end{definition} 

Here the \emph{degree} of a vertex $ x $ is the number of vertices $ y \in V_X $ such that $ (x,y) \in E_X $. 
We can visualise a graph $ X $ with a cherry as in Figure \ref{cherry}, where the box symbolises an arbitrary graph connected to the vertex $ x $ by a single edge. 

\begin{figure}[ht]
\centering
\begin{tikzpicture}
    \node[fill=black, circle, inner sep=2pt] (cherry) at (0, -1) {};
    \node at (0.4, -1) {$x$}; 
    
    \node[fill=black, circle, inner sep=2pt] (leaf1) at (2, -2) {};
    \node at (2.4, -2) {$x_1$}; 
    
    \node[fill=black, circle, inner sep=2pt] (leaf2) at (-2, -2) {};
    \node at (-2.4, -2) {$x_2$}; 
     
    \node[draw, minimum width=2cm, minimum height=1cm] (rectangle) at (0, 1) {};

    \draw (cherry) -- (leaf1);
    \draw (cherry) -- (leaf2);

    \draw (cherry) -- (rectangle);
\end{tikzpicture}
\caption{A graph with a cherry} \label{cherry}
\end{figure}
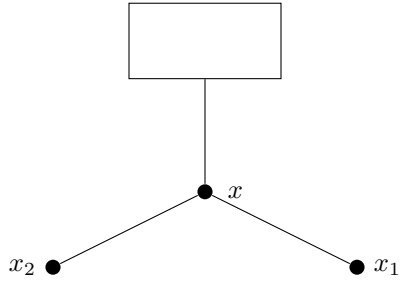

Any graph with a disjoint pair of cherries admits a pair of disjoint automorphisms given by flipping the ``branches'' of the cherries, and therefore has quantum symmetry by Proposition \ref{disjointautomorphisms}. The work by Junk-Schmidt-Weber thus establishes the following analogue of a well-known result by Erd\H{o}s-R\'enyi \cite{ERDOES_RENYI_asymmetricgraphs} about classical symmetries of finite trees.  

\begin{theorem}[Junk-Schmidt-Weber 2020]
Almost all finite trees have quantum symmetry. 
\end{theorem}

In fact, one can give a complete characterisation of the quantum automorphism groups of finite trees as follows \cite{DEBRUYN_KAR_ROBERSON_SCHMIDT_ZEMAN_quantumautomorphismtrees}. 

\begin{theorem}[de Bruyn-Kar-Roberson-Schmidt-Zeman 2023] \label{trees}
The class $ \T $ of quantum automorphism groups of finite trees is the smallest class of discrete quantum groups 
with the following properties. 
\begin{bnum} 
\item[a)] The trivial quantum group is contained in $ \T $. 
\item[b)] If $ \GG, \HH \in \T $ then their free product $ \GG * \HH $ is contained in $ \T $
\item[c)] If $ \GG \in \T $ then the free wreath product $ \GG \,wr_* \Sym_n^+ $ is contained in $ \T $. 
\end{bnum} 
\end{theorem} 

We shall not explain the definition of free products or free wreath products of discrete quantum groups here. Let us only remark that the main result of \cite{DEBRUYN_KAR_ROBERSON_SCHMIDT_ZEMAN_quantumautomorphismtrees} is phrased differently, using the language of compact 
quantum groups. It is straightforward to check that this implies Theorem \ref{trees} via discretisation. 

The proof by de Bruyn-Kar-Roberson-Schmidt-Zeman proceeds by first reducing the claim to the case of rooted trees and then using an inductive argument. 
For the details and further information we refer to \cite{DEBRUYN_KAR_ROBERSON_SCHMIDT_ZEMAN_quantumautomorphismtrees}.

\subsubsection{The infinite line} 

Let us study the quantum symmetry of the ``infinite line'' graph $ L $, that is, the graph with vertex set $ \mathbb{Z} $ and edges connecting $ n $ and $ n + 1 $ 
for all $ n \in \mathbb{Z} $. This is the (undirected) Cayley graph of the group $ \mathbb{Z} $ with respect to the generating set $ \{\pm 1\} $. 

We shall need an auxiliary result. 
For vertices $ x,y $ in a graph $ X $, denote by $ d(x,y) $ the distance between $ x $ and $ y $, that is, the length 
of a shortest path connecting them. By definition, we set $ d(x,y) = \infty $ if there is no such path. 

\begin{lemma}[Schmidt 2020] \label{distancelemma}
Let $ X $ be a graph and let $ x_1, x_2, y_1, y_2 \in V_X $ such that $ d(x_1, x_2) \neq d(y_1, y_2) $. If $ (\Hc, u) $ is a quantum automorphism of $ X $ 
then $ u_{x_1, y_1} u_{x_2, y_2} = 0 $. 
\end{lemma}

\begin{proof} 
The following argument is taken from \cite[Lemma 3.2]{SCHMIDT_distancetransitive}. Assume without loss of generality that $ m = d(x_1, x_2) < d(y_1, y_2) $. 
If $ m = 0 $ we have $ x_1 = x_2, y_1 \neq y_2 $ and hence $ u_{x_1, y_1} u_{x_1, y_2} = 0 $ 
by Lemma \ref{orthogonality}. If $ m = 1 $ we have $ (x_1, x_2) \in E_X $ and $ (y_1, y_2) \notin E_X $, in which case the claim follows from Lemma \ref{adjacencyrelations}. 
We note here that the statement and proof of this Lemma extend to infinite graphs in a straightforward way. 

It thus suffices to consider $ m \geq 2 $. Let $ x_1 = v_0, v_1, \dots, v_m = x_2 $ be a path of length $ m $ connecting $ x_1 $ and $ x_2 $. By the magic unitary condition we get 
\begin{align*}
u_{x_1, y_1} u_{x_2, y_2} &= u_{x_1, y_1} \bigg(\sum_{z_1 \in V_X} u_{v_1, z_1} \bigg) \cdots \bigg(\sum_{z_{m - 1} \in V_X} u_{v_{m - 1}, z_{m - 1}} \bigg) u_{x_2, y_2} \\
&= \sum_{z_1, \dots, z_{m - 1} \in V_X} u_{x_1, y_1} u_{v_1, z_1} \cdots u_{v_{m - 1}, z_{m - 1}} u_{x_2, y_2}, 
\end{align*} 
using for the second equality that multiplication is jointly strongly continuous on bounded sets. 
Since $ d(y_1, y_2) > m $ there is no path of length $ m $ from $ y_1 $ to $ y_2 $. This means that in each summand of the last sum, there is an 
index $ 0 \leq k \leq m - 1 $ such that $ (z_k, z_{k + 1}) \notin E_X $, where we interpret $ z_0 = y_1, z_m = y_2 $. We then have $ u_{v_k, z_k} u_{v_{k + 1}, z_{k + 1}} = 0 $, 
and hence also $ u_{x_1, y_1} u_{v_1, z_1} \cdots u_{v_{m - 1}, z_{m - 1}} u_{x_2, y_2} = 0 $. This yields the claim. 
\end{proof} 

We are now ready to discuss the following result. 

\begin{prop} \label{linegraph}
The ``infinite line'' graph has no quantum symmetry. 
\end{prop} 

\begin{proof} 
For simplicity we shall only show that the infinite line graph $ L $ has no non-classical finite dimensional quantum automorphisms. For the general case see \cite[Proposition 7.9]{VOIGT_infinitequantumpermutations}. 

Assume that $ (\Hc,u) $ is a finite dimensional quantum automorphism  of $ L $ and pick $ i,j,k,l \in \mathbb{Z} $ such that $ u_{i,j}, u_{k,l} $ do not commute. 
Due to Lemma \ref{distancelemma} we may assume without loss of generality $ k = i + m $ and $ l = j \pm m $ for some $ m > 0 $. 
Since all $ u_{k,s} $ commmute with $ u_{i,j} $ unless $ |s - j| = m $ we get in fact that $ u_{i,j} $ does not commute with either of $ u_{i + m, j \pm m} $. 

Since $ u_{i + 2m, j} $ is the only other term apart from $ u_{i,j} $ in the $ j $-th column of $ u $ whose first index has distance $ m $ from $ i + m $, 
the magic unitary condition implies that $ u_{i + 2m, j} $ does not commute with either of $ u_{i + m, j \pm m} $. Now consider the set
$$
S = \{(i + rm, j + sm) \mid r, s \in \mathbb{Z} \text{ and } r + s \text{ even} \} \subset \mathbb{Z}^2.  
$$
Using the above argument repeatedly we see that the projections $ u_{a,b}, u_{c,d} $ do not commute if $ (a,b), (c,d) \in S $ are adjacent in the sense that $ (c,d) = (a \pm m, b \pm m) $. This means in particular that all matrix entries $ u_{a,b} $ with $ (a,b) \in S $ must be nonzero. 
Hence all operators $ u_{i + 2rm, j} $ for $ r \in \mathbb{Z} $ are nonzero. 
This is impossible for a finite dimensional quantum automorphism. 
\end{proof}

\subsubsection{The Rado graph} 

Let us finally consider the Rado graph $ R $, also known as the random graph, or Erd\H{o}s-R\'enyi graph, see \cite{CAMERON_randomgraphrevisited} for a survey. 
This graph admits many equivalent descriptions, and is obtained with probability $ 1 $ if one 
takes a countable set of vertices and attaches edges to them at random. 

Explicitly, the graph $ R $ can be defined as the graph with vertex set $ V_R $ given by the prime numbers congruent $ 1 \bmod 4 $, 
and with edges $ (p,q) \in E_R $ iff $ p $ is a quadratic residue mod $ q $. Here the symmetry relation $ (p,q) \in E_R $ iff $ (q,p) \in E_R $ follows from 
the law of quadratic reciprocity. 

A key property of $ R $ is that for any pair of disjoint finite sets $ A,B \subset V_R $ there exists a vertex $ w \in V_R $ such that $ (x,w) \in E_R $ 
for all $ x \in A $ and $ (y,w) \notin E_R $ for all $ y \in B $. 
Using this property it can be shown that any partial isomorphism between finite subgraphs of $ R $ can be extended to a global automorphism. Here by a \emph{ partial isomorphism between finite subgraphs} we mean a bijection $ \theta: A \rightarrow B $ between finite subsets $ A,B \subset V_R $ which preserves the adacency relation in the sense that $ (\theta(x),\theta(y)) \in E_R \Leftrightarrow (x,y) \in E_R $ 
for all $ x,y \in A $. This implies that $ R $ has a large automorphism group.  

In contrast, we have the following result due to Ismaeel \cite{ISMAEEL_rado}.

\begin{theorem} \label{ER}
The Rado graph $ R $ has no quantum symmetry. 
\end{theorem} 

\begin{proof} 
We shall only give a proof of the simpler fact that $ R $ admits no finite dimensional quantum automorphisms which are not classical, see \cite[Proposition 7.15]{VOIGT_infinitequantumpermutations}. 

Assume that $ \sigma = (\Hc, u) $ is a non-classical finite dimensional quantum automorphism, so that the algebra generated by the projections $ u_{x, y} $ for $ x, y \in V_R $ is noncommutative. 

By assumption, we then find vertices $ x(+), x(-) $ and $ y_1(+), y_1(-) $ such $ u_{x(+), y_1(+)} $ and $ u_{x(-), y_1(-)} $ do not commute. 
We let $ u_{x(\pm), y_j(\pm)} $ for $ j = 2, \dots, r_\pm $ be the remaining nonzero projections in the row for $ x(\pm) $. Set 
$$
p_\alpha(\pm) = u_{x(\pm), y_1(\pm)}, \qquad p_\beta(\pm) = \sum_{j > 1} u_{x(\pm),y_j(\pm)}.  
$$
Then the projections $ p_\alpha(+), p_\beta(+) $ are orthogonal with $ p_\alpha(+) + p_\beta(+) = 1 $, 
and in the same way $ p_\alpha(-), p_\beta(-) $ are orthogonal with $ p_\alpha(-) + p_\beta(-) = 1 $. 
By construction, $ p_\alpha(+) $ and $ p_\alpha(-) $ do not commute, and hence $ p_\beta(+) $ and $ p_\beta(-) $ do not commute either. 

Choose a vertex $ w $ such that $ (w, y_1(+)) \in E_R, (w,y_1(-)) \in E_R $ and $ (w, y_i(\pm)) \notin E_R $ for all $ i > 1 $,
and consider the nonzero projections $ u_{v_i, w} $ in the column for $ w $, where $ 1 \leq i \leq k $ for some $ k $. 
If $ (v_i, x(\pm)) \in E_R $ then $ u_{v_i, w} $ is orthogonal to $ p_\beta(\pm) $, 
and if $ (v_i, x(\pm)) \notin E_R $ then $ u_{v_i, w} $ is orthogonal to $ p_\alpha(\pm) $. 
We conclude $ u_{v_i,w} \leq p_\alpha(\pm) $ in the first case and $ u_{v_i, w} \leq p_\beta(\pm) $ in the second case. 
Since the projections $ u_{v_i, w} $ form a partition of unity it follows that we can write each of $ p_\alpha(\pm) $ and $ p_\beta(\pm) $ as sums of certain $ u_{v_i, w} $. 

However, this means in particular that all these projections commute, which yields a contradiction. 
\end{proof}

\subsection{Exercises}

\begin{exercise}
Let $ X $ be a graph. The complement of $ X $ is the graph $ X^c $  with the same vertex set as $ X $, such that for $ x,y \in V_X $ with $ x \neq y $ we have $ (x,y) \in E_{X^c} $ iff $ (x,y) \notin E_X $. 
Show that a quantum permutation $ (\Hc,u) $ of $ V_X $ is a quantum automorphism of $ X $ iff it is a quantum automorphism of $ X^c $. 
\end{exercise}

\begin{exercise}
Let $ X $ be a connected graph with no quantum symmetry and let $ Y $ be the disjoint union of two copies of $ X $. Does $ Y $ have quantum symmetry? 
\end{exercise}

\begin{exercise}
In the setting of the previous exercise, what happens if we no longer assume that $ X $ is connected? 
\end{exercise}

\begin{exercise} 
A graph $ X $ is called locally finite if every vertex $ x \in V_X $ is connected to only finitely many other vertices. In this case, the adjacency matrix $ A_X $ of $ X $ is locally finite, that is, there are only finitely many nonzero entries in each column and each row of $ A_X $. 

By definition, the coherent algebra of a locally finite graph $ X $ is the smallest 
subalgebra of locally finite $ V_X \times V_X $-matrices containing $ \id $ and $ A_X $ which is closed under entrywise multiplication. 
Show that if $ (\Hc, u) $ is a quantum automorphism of $ X $ then $ T u = u T $ for all $ T $ in the coherent algebra of $ X $. 
\end{exercise}

\begin{exercise}
Describe the irreducible quantum automorphisms of all graphs with at most $ 4 $ vertices. 
\end{exercise}

\section{Quantum symmetries in topology} 

In this final chapter we study quantum symmetry in the setting of locally compact topological spaces. This builds on the theory of infinite quantum permutations, by incorporating a natural notion of continuity. 

\subsection{Definitions} 

Let $ X $ be a locally compact space and denote by $ X_\delta $ the underlying set, or equivalently, the space $ X $ equipped 
with the discrete topology. Then the identity map $ X_\delta \rightarrow X $ is continuous and induces a 
nondegenerate injective $ * $-homomorphism $ can: C_0(X) \rightarrow M(C_0(X_\delta)) = C_b(X_\delta) $. 
Note that, in fact, every set-theoretic map $ X \rightarrow Y $ from $ X $ into a topological space $ Y $ becomes continuous when viewed as 
a map $ X_\delta \rightarrow Y $. 

We are interested in quantum permutations of $ X_\delta $ which are compatible with the topology 
of $ X $ in a suitable sense. Note that if $ \sigma = (\Hc_\sigma, u^\sigma) $ is a quantum permutation of $ X_\delta $ in the sense of Definition \ref{defquantumpermutation}, then we 
obtain an associated  nondegenerate $ * $-homomorphism $ \alpha_\delta: C_0(X_\delta) \rightarrow M(K(\Hc_\sigma) \otimes C_0(X_\delta)) $ by setting 
$$ 
\alpha_\delta(e_x) = \sum_{y \in X_\delta} u^\sigma_{x, y} \otimes e_y, 
$$ 
where $ e_x $ denotes the characteristic function based at $ x $. Here $ K(\Hc_\sigma) $ denotes the algebra of compact operators on $ \Hc_\sigma $.  

\begin{definition} \label{qhomeo}
Let $ X $ be a locally compact space. A quantum permutation $ \sigma = (\Hc_\sigma, u^\sigma) $ of $ X_\delta $ is called continuous 
if there exists a nondegenerate $ * $-homomorphism $ \alpha: C_0(X) \rightarrow M(K(\Hc_\sigma) \otimes C_0(X)) $ making the diagram 
\begin{center}
\begin{tikzcd}
C_0(X) \arrow[r, "\alpha"]\arrow[d, swap, "can"] & M(K(\Hc_\sigma) \otimes C_0(X)) \arrow[d, "\id \otimes can"] \\
M(C_0(X_\delta)) \arrow[r, "\alpha_\delta"] & M(K(\Hc_\sigma) \otimes C_0(X_\delta)) 
\end{tikzcd}
\end{center}
commutative and $ [\alpha(C_0(X))(K(\Hc_\sigma) \otimes 1)] = K(\Hc_\sigma) \otimes C_0(X) $. \\
A quantum homeomorphism of $ X $ is a continuous quantum permutation $ \sigma $ of $ X_\delta $ such that its conjugate $ \overline{\sigma} $ is  
a continuous quantum permutation as well. 
\end{definition} 

Of course, if $ X $ carries the discrete topology then quantum homeomorphisms of $ X $ are nothing but quantum permutations in the sense of Definition \ref{defquantumpermutation}. 

For general $ X $, it follows from the fact that the map $ can $ is injective that continuity of $ (\Hc_\sigma, u^\sigma) $ is equivalent to saying that the image of $ C_0(X) $ under the map $ \alpha_\delta $ is contained in the subalgebra of 
$$ 
M(K(\Hc_\sigma) \otimes C_0(X)) = C_b(X, B(\Hc_\sigma)_{st}) \subset C_b(X_\delta, B(\Hc_\sigma)) = M(K(\Hc_\sigma) \otimes C_0(X_\delta))
$$ 
consisting of all $ g \in C_b(X, B(\Hc_\sigma)_{st}) $ such that $ g(T \otimes 1), (T \otimes 1)g \in C_0(X, K(\Hc_\sigma)) $ for all $ T \in K(\Hc_\sigma) $, and $ [\alpha_\delta(C_0(X))(K(\Hc_\sigma) \otimes 1)] = K(\Hc_\sigma) \otimes C_0(X) $. 
Here we write $ B(\Hc_\sigma)_{st} $ for $ B(\Hc_\sigma) $ equipped with the strict topology. 

Under additional assumptions the above description can be simplified. If $ \Hc_\sigma $ is finite dimensional 
then the strict topology on $ B(\Hc_\sigma) = K(\Hc_\sigma) $ agrees with the norm topology, and the target of $ \alpha $ becomes $ B(\Hc_\sigma) \otimes C_0(X) $. If in addition $ X $ is compact, then the top arrow in Definition \ref{qhomeo}
is a unital $ * $-homomorphism $ \alpha: C(X) \rightarrow B(\Hc_\sigma) \otimes C(X) \cong M_n(C(X)) $, where $ n = \dim(\Hc_\sigma) $. 

\begin{lemma} \label{checkingqhomeo}
Let $ X $ be a locally compact space and let $ \sigma = (\Hc_\sigma, u^\sigma) $ be a finite dimensional quantum permutation of $ X_\delta $. Then the 
following conditions are equivalent. 
\begin{bnum}
\item[a)] $ \sigma $ is a quantum homeomorphism. 
\item[b)] For every $ f \in C_0(X) $ the maps 
\begin{align*} 
X \ni y &\mapsto\sum_{x \in X} f(x) u^\sigma_{x,y} \in B(\Hc_\sigma) \\
X \ni x &\mapsto \sum_{y \in X} f(y) u^\sigma_{x,y} \in B(\Hc_\sigma) 
\end{align*} 
are contained in $ C_0(X, B(\Hc_\sigma)) = B(\Hc_\sigma) \otimes C_0(X) $. 
\end{bnum}
\end{lemma} 

\begin{proof} 
$ a) \Rightarrow b) $ By assumption, $ \alpha(f)(y) = \sum_x f(x) u^\sigma_{x,y} $ defines an element $ \alpha(f) \in C_0(X, B(\Hc_\sigma)) $ for every $ f \in C_0(X) $. Applying the same argument to $ \overline{\sigma} $ yields the claim for the second map, using the $ * $-anti-isomorphism $ j: B(\Hc_\sigma) \rightarrow B(\overline{\Hc_\sigma}) $ given by $ j(T) (\overline{\xi}) = \overline{T^* \xi} $. \\
$ b) \Rightarrow a) $ The condition for the first map in $ b) $ means that we obtain a well-defined $ * $-homomorphism $ \alpha: C_0(X) \rightarrow B(\Hc_\sigma) \otimes C_0(X) = C_0(X, B(\Hc_\sigma)) $ 
given by the formula $ \alpha(f)(y) = \sum_x f(x) u^\sigma_{x,y} $. In order to show $ [\alpha(C_0(X)) (B(\Hc_\sigma) \otimes 1)] = B(\Hc_\sigma) \otimes C_0(X) $, it clearly suffices to verify $ 1 \otimes C_0(X) \subset [\alpha(C_0(X)) (B(\Hc_\sigma) \otimes 1)] $. Let $ f \in C_0(X) $ and choose an orthonormal basis $ e_1, \dots, e_n $ of $ \Hc_\sigma $. The condition for the second map in $ b) $ implies that  
$$
f_{i,j}(x) = \sum_z f(z) \bra e_i, u^\sigma_{x,z} e_j \ket
$$
is contained in $ C_0(X) $ for all $ 1 \leq i,j \leq n $. We calculate 
\begin{align*}
\sum_{i,j} \alpha(f_{i,j})(y) |e_i \ket \bra e_j| &= \sum_{i,j} \sum_x f_{i,j}(x) u^\sigma_{x,y} |e_i \ket \bra e_j| \\
&= \sum_{i,j} \sum_{x,z} f(z) \bra e_i, u^\sigma_{x,z} e_j \ket u^\sigma_{x,y} |e_i \ket \bra e_j| \\
&= \sum_j \sum_{x,z} f(z) u^\sigma_{x,y} u^\sigma_{x,z} |e_j \ket \bra e_j| \\
&= \sum_{x} f(y) u^\sigma_{x,y} = f(y) 1,  
\end{align*}
which yields the claim. It follows that $ \sigma $ is continuous. 

For the conjugate quantum permutation $ \overline{\sigma} $ one argues in a completely analogous way. 
\end{proof} 

Using Lemma \ref{checkingqhomeo} it is straightforward to check that a one-dimensional quantum permutation of $ X_\delta $ is a quantum homeomorphism of $ X $ iff the corresponding classical permutation is a homeomorphism of $ X $ in the usual sense. 
We note that the maps $ x \mapsto u^\sigma_{x,y} $ and $ x \mapsto u^\sigma_{y,x} $ for fixed $ y \in X $ typically fail to be (strictly) continuous for a quantum homeomorphism $ \sigma = (\Hc_\sigma, u^\sigma) $, even if $ \Hc_\sigma $ is finite dimensional. 

\begin{prop} \label{qhomeocategory}
Let $ X $ be a locally compact space. The collection of all quantum homeomorphisms of $ X $ defines a full $ C^* $-tensor subcategory $ \HOMEO^+(X) $ 
of the category $ \PERM^+(X_\delta) $ of quantum permutations of $ X_\delta $. 
\end{prop} 

\begin{proof} 
Let us verify that the tensor product of two continuous quantum permutations $ \sigma = (\Hc_\sigma, u^\sigma), \tau = (\Hc_\tau, u^\tau) $ of $ X $ is again continuous. 
By assumption, there exist nondegenerate $ * $-homomorphisms $ \alpha: C_0(X) \rightarrow M(K(\Hc_\sigma) \otimes C_0(X)), \beta: C_0(X) \rightarrow M(K(\Hc_\tau) \otimes C_0(X)) $ 
with the properties in Definition \ref{qhomeo}. Then $ \gamma = (\id \otimes \beta)\alpha $ is a nondegenerate $ * $-homomorphism from $ C_0(X) $ into $ M(K(\Hc_\sigma) \otimes K(\Hc_\tau) \otimes C_0(X)) $. Noting that 
$$
\gamma(f)(z) = (\id \otimes \beta)\alpha(f)(z)  
=\sum_{x,y \in X} f(x) u^\sigma_{x, y} \otimes u^\tau_{y, z} 
$$
for $ f \in C_0(X) $ and 
\begin{align*}
[\gamma(C_0(X)) &(K(\Hc_\sigma \otimes \Hc_\tau) \otimes 1)] = [\gamma(C_0(X)) (K(\Hc_\sigma) \otimes K(\Hc_\tau) \otimes 1)] \\
&= [(\id \otimes \beta)(\alpha(C_0(X)) (K(\Hc_\sigma) \otimes 1))(1 \otimes K(\Hc_\tau) \otimes 1)] \\
&= [(\id \otimes \beta)(K(\Hc_\sigma) \otimes C_0(X))(1 \otimes K(\Hc_\tau) \otimes 1)] \\
&= K(\Hc_\sigma) \otimes K(\Hc_\tau) \otimes C_0(X) = K(\Hc_\sigma \otimes \Hc_\tau) \otimes C_0(X),
\end{align*}
we see that $ \gamma $ provides the required $ * $-homomorphism for $ \sigma \tp \tau $. 
Applying this argument to a pair of quantum homeomorphisms and their conjugates shows that the class of quantum homeomorphisms is preserved under tensor products. 

The identity quantum permutation of $ X $ is clearly a quantum homeomorphism. 
Similarly, it is not hard to check that direct sums and subobjects of quantum homeomorphisms are again quantum homeomorphisms.
\end{proof} 

By definition, the conjugate of a quantum homeomorphism is again a quantum homeomorphism. 
Hence it follows from Proposition \ref{qhomeocategory} that the full subcategory $ \Homeo^+(X) $ of $ \HOMEO^+(X) $ consisting of all finite dimensional quantum homeomorphisms is a rigid $ C^* $-tensor category. In combination with the Tannaka-Krein theorem \ref{TK}, this means that $ \Homeo^+(X) $ determines a quantum subgroup $ \Qut_\delta(X) \subset \Sym^+(X_\delta) $. 

\begin{definition} 
Let $ X $ be a locally compact space. We call $ \Qut_\delta(X) $ the quantum homeomorphism group of $ X $. 
\end{definition} 

As already indicated above, the discrete quantum group $ \Qut_\delta(X) $ contains the group $ \Homeo(X) $ of classical homeomorphisms as a quantum subgroup. 

At this stage, given a locally compact space $ X $, it is natural to ask whether there exist quantum homeomorphisms which are not classical, that is, whether $ X $ has quantum symmetry. The following argument, modelled on Proposition \ref{disjointautomorphisms}, shows that there is typically an 
abundance of non-classical quantum homeomorphisms. 
Let us say that two classical homeomorphisms $ \sigma, \tau $ of a locally compact space $ X $ are \emph{disjoint} if $ \sigma(x) \neq x \implies \tau(x) = x $, or equivalently, $ \tau(x) \neq x \implies \sigma(x) = x $ for all $ x \in X $. 

\begin{prop} \label{disjointhomeomorphisms}
Let $ X $ be a locally compact space admitting a pair of disjoint homeomorphisms $ \sigma, \tau \in \Homeo(X) $, and assume that $ k \in \mathbb{N} $ does not exceed the order
of either of them. Then $ X $ admits an irreducible quantum homeomorphism of dimension $ k $. In particular, if $ \Homeo(X) $ contains a pair of nontrivial 
disjoint homeomorphisms then $ X $ has quantum symmetry. 
\end{prop} 

\begin{proof} 
Using the same notation and arguments as in the proof of Proposition \ref{disjointautomorphisms} we obtain an irreducible quantum permutation $ u^\rho $ 
of $ X_\delta $ of dimension $ k $ by setting 
$$ 
u^\rho_{x,y} = \sum_{r = 1}^k \delta_{x, \sigma^r(y)} p_r + \sum_{s = 1}^k \delta_{x, \tau^s(y)} q_s - \delta_{x,y} 1
$$ 
for $ x, y \in X $. The only thing that needs to be checked is that $ u^\rho $ is a quantum homeomorphism. To this end let $ f \in C_0(X) $ and note that 
\begin{align*}
\sum_{y \in X} f(y) u^\rho_{x,y} &= \sum_{y \in X} f(y) \bigg(\sum_{r = 1}^k \delta_{x, \sigma^r(y)} p_r + \sum_{s = 1}^k \delta_{x, \tau^s(y)} q_s - \delta_{x,y} 1 \biggr) \\
&= \sum_{r = 1}^k f(\sigma^{-r}(x)) p_r + \sum_{s = 1}^k f(\tau^{-s}(x)) q_s - f(x) 1, 
\end{align*}
and similarly 
\begin{align*}
\sum_{x \in X} f(x) u^\rho_{x,y} &= \sum_{r = 1}^k f(\sigma^r(y)) p_r + \sum_{s = 1}^k f(\tau^s(y)) q_s - f(y) 1. 
\end{align*}
Hence the claim follows from Lemma \ref{checkingqhomeo}.  
\end{proof} 

Proposition \ref{disjointautomorphisms} allows one to exhibit quantum symmetries in various concrete situations. Consider for instance the circle $ X = S^1 $,  
and let $ \tau, \sigma \in \Homeo(S^1) $ be homeomorphisms which move opposite halves of the circle. Concretely, upon viewing $ S^1 $ as the unit interval with the endpoints 
identified, we can take $ \tau $ and $ \sigma $ to be the piecewise linear homeomorphisms 
\begin{align*}
\sigma(x) &=
\begin{cases}
\frac{1}{2}x & \text{if } 0 \leq x \leq \frac{1}{4} \\
\frac{3}{2}x - \frac{1}{4} & \text{if } \frac{1}{4} \leq x \leq \frac{1}{2} \\
x & \text{if } \frac{1}{2} \leq x \leq 1 \\
\end{cases}
\\
\tau(x) &=
\begin{cases}
x & \text{if } 0 \leq x \leq \frac{1}{2} \\
\frac{1}{2}x + \frac{1}{4} & \text{if } \frac{1}{2} \leq x \leq \frac{3}{4} \\
\frac{3}{2}x - \frac{1}{2} & \text{if } \frac{3}{4} \leq x \leq 1
\end{cases}
\end{align*}
which may be visualised as in Figure \ref{disjoint}. 

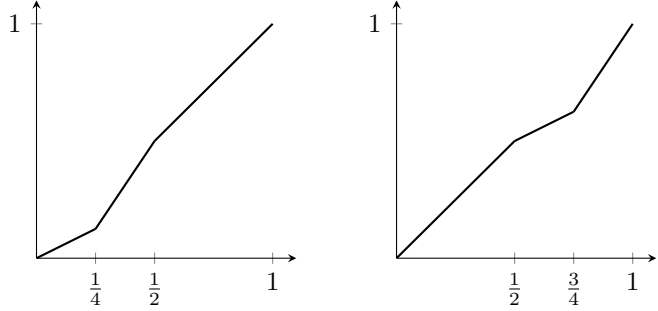
\begin{figure}[ht]
\centering
\begin{tikzpicture}
  \begin{axis}[
      axis lines = middle,
      xlabel = $$, ylabel = $$,
      xmin=0, xmax=1.1,
      ymin=0, ymax=1.1,
      samples=100,
      width=5cm,
      height=5cm,
      grid=none,
      xtick={0, 0.25, 0.5,1}, 
      xticklabels={$0$, $\frac{1}{4}$, $\frac{1}{2}$, $1$}, 
      ytick={1}, 
      yticklabels={1}, 
      yticklabel style={left}, 
     ]
    \addplot[black, thick, domain=0:0.25] {0.5* x};
    \addplot[black, thick, domain=0.25:0.5] {1.5*x - 0.25};
    \addplot[black, thick, domain=0.5:1] {x};
  \end{axis}
\end{tikzpicture}
\qquad
\begin{tikzpicture}
  \begin{axis}[
      axis lines = middle,
      xlabel = $$, ylabel = $$,
      xmin=0, xmax=1.1,
      ymin=0, ymax=1.1,
      samples=100,
      width=5cm,
      height=5cm,
      grid=none,
      xtick={0, 0.5, 0.75,1}, 
      xticklabels={$0$, $\frac{1}{2}$, $\frac{3}{4}$, $1$}, 
      ytick={1}, 
      yticklabels={1}, 
      yticklabel style={left}, 
    ]
    \addplot[black, thick, domain=0:0.5] {x};
    \addplot[black, thick, domain=0.5:0.75] {0.5*x + 0.25};
    \addplot[black, thick, domain=0.75:1] {1.5*x - 0.5};
  \end{axis}
\end{tikzpicture}
\caption{The homeomorphisms $ \sigma $ and $ \tau $} 
\label{disjoint}
\end{figure} 

Using similar arguments one finds that ``almost all'' locally compact spaces have quantum symmetry in the sense described above. In particular, every manifold of dimension greater than $ 0 $ has quantum symmetry. 

This is in marked contrast to a result by Goswami \cite{GOSWAMI_nonexistence} which asserts that connected Riemannian manifolds do not 
admit non-classical quantum \emph{isometries}, that is, quantum symmetries which are smooth and compatible with the Riemannian metric in a suitable sense. 
For the definition of the quantum isometry group of a Riemannian manifold and more background we refer the reader to \cite{GOSWAMI_nonexistence} and the references therein.

\subsection{Covering spaces} 

We have seen above that the quantum homeomorphism group of a locally compact space is typically huge, and it is not realistic to aim for an explicit description of all quantum homeomorphisms. However, in a similar way as in our discussion in section \ref{sec_infinite}, interesting questions arise once we add further structure to the picture. 

Recall that a \emph{covering space} is a pair of topological spaces $ E,X $ together with a continuous map $ p: E \rightarrow X $ such that each point $ x \in X $ 
has an open neighborhood $ U $ such that $ E_{|U} = p^{-1}(U) \cong \bigsqcup_{i \in I} U $ identifies with a disjoint union of copies of $ U $ and the restriction $ p_{|U}: E_{|U} \rightarrow U $ 
of $ p $ to $ E_{|U} $ is the canonical projection,  see for instance \cite[Section 1.3]{HATCHER_book}. We will sometimes use the short-hand notation $ E|X $ for a covering space, suppressing the covering map $ p $. 

A \emph{deck transformation} of $ p: E \rightarrow X $ is a homeomorphism $ T: E \rightarrow E $ such that $ p \circ T = p $. 
In the sequel we will always assume that $ X $ is connected and locally compact. In this case, the total space $ E $ is locally compact as well, 
and all fibers $ E_x = p^{-1}(x) $ for $ x \in X $ have the same cardinality, called the \emph{number of sheets} of $ E|X $. 

\begin{definition} 
Let $ p: E \rightarrow X $ be a covering space. A quantum homeomorphism $ \sigma = (\Hc, u) $ of $ E $ is called a quantum deck transformation 
if $ u_{x, y} = 0 $ unless $ p(x) = p(y) $. 
\end{definition} 

Every quantum deck transformation $ \sigma $ induces quantum permutations $ \sigma_{|x} $ of the fibres $ E_x $ for $ x \in X $, and  
restricts to a quantum homeomorphism $ \sigma_{|U}: E_{|U} \rightarrow E_{|U} $ for every open set $ U \subset X $. 

The \emph{tautological covering} of $ X $ is given by $ E = X $ and $ p = \id $, and the following observation follows directly from the definition. 

\begin{example} \label{trivialcovering}
Every quantum deck transformation $ (\Hc, u) $ of a tautological covering is of the form $ u_{x,y} = \delta_{x,y} \id $. 
\end{example} 

More generally, a covering of the form $ E = X \times I \rightarrow X $, where $ I $ is some index set the covering map given by projecting onto the first factor, is called the \emph{trivial covering} of $ X $ with fibre $ I $. 

\begin{example} \label{exconstant}
Consider the trivial covering $ E = X \times I \rightarrow X $ with fibre $ I $. Then every quantum permutation $ (\Hc, u) $ of $ I $ induces a quantum deck transformation $ (\Hc, u^E) $ of $ E|X $ by setting 
$$
u^E_{(x, i), (y,j)} = \delta_{x,y} u_{i,j}
$$
for $ (x,i), (y,j) \in E $. We shall call quantum deck transformations of this form constant.
\end{example}

In contrast to the case of classical deck transformations, not every quantum deck transformation of a trivial covering needs to be constant, even if all its restrictions to the fibres are classical quantum permutations. Before explaining this, let us discuss a useful criterion characterising finite dimensional quantum deck transformations. 

\begin{lemma} \label{checkingquantumdeck}
Let $ p: E \rightarrow X $ be a covering space and let $ \sigma = (\Hc, u) $ be a finite dimensional quantum permutation of $ E_\delta $ satisfying $ u_{x, y} = 0 $ unless $ p(x) = p(y) $. 
Then  $ \sigma $ is a quantum deck transformation iff for every $ x \in X $ there exists an open neighborhood $ U $ 
such that $ E_{|U} \cong \bigsqcup_{i \in I} U = U \times I $ is trivial and the maps 
\begin{align*}
U \ni y \mapsto u_{(y,i)(y,j)} \in B(\Hc) \\
\end{align*}
are continuous for all $ i,j \in I $. 
\end{lemma} 

\begin{proof} 
Assume first that $ \sigma $ is a quantum deck transformation of $ E|X $. Moreover let $ x \in X $, and let $ V $ be an open neighborhood of $ x $ over which $ E $ is trivial. Pick an open neighborhood $ U \subset V $ of $ x $ such that the closure $ \overline{U} $ is compact and contained in $ V $. Then we can find a 
function $ f \in C_0(E) $ with support in the $ i $-th copy of $ V $ inside $ E_{|V} \cong V \times I $ such that $ f = 1 $ on $ U \subset V $. If we write $ \alpha: C_0(E) \rightarrow B(\Hc_\sigma) \otimes C_0(E) = C_0(E, B(\Hc_\sigma)) $ for the $ * $-homomorphism associated to $ \sigma $, then 
$ \alpha(f)(e) $ is nonzero only for $ e \in E_{|V} $, and we get $ \alpha(f)(y, j) = f(y,i) u_{(y,i),(y,j)} = u_{(y,i),(y,j)} $ for all $ y \in U $. It follows that $ U \ni y \mapsto u_{(y,i),(y,j)} $ is continuous 

Conversely, assume that the quantum permutation $ \sigma $ satisfies $ u_{x,y} = 0 $ unless $ p(x) = p(y) $, such that locally one obtains  
continuous maps $ U \ni y \mapsto u_{(y,i)(y,j)} \in B(\Hc) $. Then the maps 
$$ 
(y,j) \mapsto \sum_{i \in I} f(y,i) u_{(y,i), (y, j)}, \qquad (y,i) \mapsto \sum_{j \in I} f(y,j) u_{(y,i), (y, j)} 
$$ 
are continuous on $ E_{|U} $ for all $ f \in C_0(E) $, and it follows using Lemma \ref{checkingqhomeo} that $ (\Hc, u) $ is a quantum homeomorphism. 
\end{proof} 

Let us now give one of the most basic examples of a trivial covering which admits non-constant quantum deck transformations. 

\begin{example} \label{extrivialquantumdeck}
Let $ X = S^1 = \mathbb{R}/2 \pi \mathbb{Z} $, and let $ E = X \times \{0,1\} = X \times \mathbb{Z}/2 \mathbb{Z} $ be the trivial two-sheeted covering of $ X $. Moreover let $ p_0, p_1 \in M_2(\mathbb{C}) $ be the standard projections
\begin{align*}
p_0 &= \begin{pmatrix} 
1 & 0 \\
0 & 0
\end{pmatrix}, 
\qquad 
p_1 = \begin{pmatrix} 
0 & 0 \\
0 & 1
\end{pmatrix}, 
\end{align*}
and let $ \gamma: X \rightarrow U(2) $ be the map
$$
\gamma(t) = 
\begin{pmatrix} 
\cos(t) & \sin(t) \\
-\sin(t) & \cos(t) 
\end{pmatrix}.
$$
Using Lemma \ref{checkingquantumdeck}, it is easy to check that we obtain a $ 2 $-dimensional non-constant quantum deck transformation $ (\mathbb{C}^2, u) $ of $ E|X $ by defining 
$$
u_{(x,i),(y, j)} = \delta_{x,y} \gamma(x)^* p_{i - j} \gamma(x)
$$
for all $ x, y \in X, i,j \in \mathbb{Z}/2 \mathbb{Z} $. This quantum deck transformation is non-classical, since its restrictions to the fibres at $ a = 0 $ and $ b = \pi/4 $ are  
$$
u_{|a} = \begin{pmatrix} 
p_0 & p_1 \\
p_1 & p_0 
\end{pmatrix}, 
\qquad  
u_{|b} = \begin{pmatrix} 
p_+ & p_- \\
p_- & p_+ 
\end{pmatrix},
$$
where 
\begin{align*}
p_+ = \frac{1}{2} \begin{pmatrix} 
1 & 1 \\
1 & 1
\end{pmatrix}, 
\qquad 
p_- = 
\frac{1}{2} \begin{pmatrix} 
1 & -1 \\
-1 & 1
\end{pmatrix},
\end{align*} 
respectively. 
\end{example}

Note that the quantum deck transformation in Example \ref{extrivialquantumdeck} is classical in each fibre, while it fails to be classical globally. This phenomenon is not specific to trivial coverings, and we can easily modify the above construction to nontrivial coverings as well. 

\begin{example} \label{twosheeted}
Consider the nontrivial two-sheeted covering $ E $ of $ X = S^1 $, which we can view as the map $ p(z) = z^2 $ for $ z \in E = X = U(1) \subset \mathbb{C} $. 
Moreover let $ \gamma: X \cong \mathbb{R}/2 \pi \mathbb{Z} \rightarrow U(2) $ and $ p_0, p_1 $ as in Example \ref{extrivialquantumdeck}. 
Then we obtain a $ 2 $-dimensional quantum deck transformation $ (\mathbb{C}^2, u) $ of $ E $ by setting
$$
u_{e,f} = \gamma(p(e))^* (\delta_{e,f} p_0 \oplus \delta_{e, -f} p_1) \gamma(p(e))
$$
for $ e, f \in E = U(1) $. 
Note here that the restriction $ u_{|x} \in M_2(B(\mathbb{C}^2)) $ to the fibre $ E_x \cong \{0,1\} $ over any point $ x \in X $ is preserved under flipping both rows and columns at the same time, so that this construction is indeed compatible with the structure of the covering. 
\end{example} 

Although the quantum deck transformations obtained via the method in Examples \ref{extrivialquantumdeck} and \ref{twosheeted} are non-classical, they are clearly closely related to direct sums of classical deck transformations. 
In fact, they are \emph{deck equivalent} to classical quantum deck transformation in the following sense. 

\begin{definition} 
Let $ p: E \rightarrow X $ be a covering space and let $ \sigma = (\Hc_\sigma, u^\sigma), \tau = (\Hc_\tau, u^\tau) $ be quantum deck transformations of $ E|X $. 
A (unitary) deck intertwiner from $ \sigma $ to $ \tau $ is a strictly continuous bounded map $ T: X \rightarrow B(\Hc_\sigma, \Hc_\tau) $ such that ($ T_x $ is unitary for all $ x \in X $ and)
$$
T_{p(e)} u^\sigma_{e,f} = u^\tau_{e,f} T_{p(e)} 
$$
for all $ e, f \in E $. 
Two quantum deck transformations $ \sigma = (\Hc_\sigma, u^\sigma), \tau = (\Hc_\tau, u^\tau) $ 
are called deck equivalent if there exists a unitary deck intertwiner between them.  
\end{definition} 

By construction, the map $ \gamma $ in Examples \ref{extrivialquantumdeck} and \ref{twosheeted} implements unitary deck intertwiners between the given quantum deck transformations and classical quantum deck transformations. 

We obtain a $ C^* $-tensor category $ \DECK^+(E|X) $ consisting of all quantum deck transformations of $ E|X $ and deck intertwiners between them as morphisms. 
Inside this category we have the full $ C^* $-tensor subcategory $ \Deck^+(E|X) $ of finite dimensional quantum deck transformations, and this category is rigid. Following the same pattern as before, we can thus give the following definition. 

\begin{definition} \label{defqdeck}
The quantum deck transformation group $ \Qut_\delta(E|X) $ of a covering space $ E|X $ is the discrete quantum group associated to the rigid $ C^* $-tensor category $ \Deck^+(E|X) $ of finite dimensional quantum deck transformations of $ E|X $. 
\end{definition} 

In view of Lemma \ref{quantumpermutationssmalln} and our examples so far, one might wonder whether quantum deck transformations for coverings with at most three sheets are always deck equivalent to direct sums of classical quantum deck transformations. This is not the case, as the following example shows.

\begin{example} \label{extwosphere}
Let $ X = S^2 $ be the $ 2 $-sphere. Moreover let $ L \rightarrow X $ and $ L^\perp \rightarrow X $ be nontrivial complex line bundles such that $ L \oplus L^\perp $ is trivial. For instance, viewing $ X $ as $ \mathbb{C}P^1 $ one may take $ L $ to be the tautological line bundle and $ L^\perp $ its dual. 
Choosing a trivialisation $ L \oplus L^\perp \cong X \times \mathbb{C}^2 $ gives rise to projections $ p, q \in C(S^2, M_2(\mathbb{C}))$ such that $ p + q = 1 $ and the vector bundles corresponding to $ p $ and $ q $ are isomorphic to $ L $ and $ L^\perp $, respectively. 
Define a quantum permutation $ \sigma = (\mathbb{C}^2, u) $ of the trivial two-sheeted covering $ E = X \times \mathbb{Z}/2 \mathbb{Z} $ by setting 
$$
u_{(x,i), (y,j)} = \delta_{x,y}(\delta_{i,j} p(x) + \delta_{i, 1 - j} q(x)).  
$$
Using Lemma \ref{checkingquantumdeck}, one checks easily that this is indeed a quantum deck transformation of $ E|X $. However, there is no unitary deck intertwiner from $ \sigma $ to a direct sum of classical deck transformations of $ E|X $, because such an intertwiner would give rise to trivialisations of $ L $ and $ L^\perp $. 
\end{example}

Let us next consider examples of quantum deck transformations which are non-classical in every fibre. For a trivial covering with at least four sheets, such examples can be easily obtained from non-classical quantum permutations of the fibre by taking the associated constant quantum deck transformations as in Example \ref{exconstant}.  

In order to produce examples for non-trivial coverings, we shall first explain in which sense quantum permutations, which by definition are ``generalised symmetries'', can themselves have symmetries.  

\begin{definition} \label{gequiv}
Let $ X $ be a set and let $ G $ be a classical discrete group acting on $ X $ from the right. We say that a quantum permutation $ (\Hc,u) $ of $ X $ is $ G $-equivariant 
if there exists a projective unitary representation $ \theta: G \rightarrow U(\Hc) $ such that 
$$
u_{x \cdot g, y \cdot g} = \theta_g^* u_{x, y} \theta_g 
$$
for all $ x, y \in X $ and $ g \in G $. 
\end{definition}  

We note that asking for strict equality $ u_{x \cdot g, y \cdot g} = u_{x,y} $ in Definition \ref{gequiv} would be too restrictive. For instance, if $ G $ acts transitively on $ X $, then this stronger condition would already force $ (\Hc, u) $ to be classical, see Exercise \ref{ex142}.

In contrast, allowing nontrivial projective unitary representations gives us enough flexibility to obtain non-classical examples of $ G $-equivariant quantum permutations even when the action of $ G $ on $ X $ is transitive. 
Specifically, recall our description of the Weyl quantum permutations associated with a finite abelian group given in Proposition \ref{weylquantumpermutation}. 

\begin{lemma} \label{weylequivariant}
Let $ A $ be an abelian group of order $ N $ and set $ n = N^2 $. Moreover view $ A \times A $ as an $ A \times A $-set with the action by translation. Then the Weyl quantum permutations $ \pi^g = (\mathbb{C}^n, u^g) $ for $ g \in U(N) $ are $ A \times A $-equivariant quantum permutations of $ A \times A $,  with respect to the unitary representation $ \theta $ of $ A \times A $ on $ M_N(\mathbb{C}) $ given by 
$$
\theta_{(r,s)}(T) = W_{r,s}^* T W_{r,s} 
$$
for $ (r,s) \in A \times A $. 
\end{lemma} 

\begin{proof} 
It follows from Lemma \ref{weyllemma} $ b) $ that $ \theta $ is a unitary representation. 
Moreover, for $ (a,b),(c,d),(r,s) \in A \times A $ and $ g \in U(N) $ we calculate 
\begin{align*}
&u^g_{(a+r,b+s),(c+r,d+s)} = |W_{a+r,b+s} g W_{c+r,d+s}^* \ket \bra W_{a+r,b+s} g W_{c+r,d+s}^*| \\
&= |\overline{[r,b]} [r, d] W_{r,s} W_{a,b} g W_{c,d}^* W_{r,s}^* \ket \bra \overline{[r,b]} [r, d] W_{r,s} W_{a,b} g W_{c,d}^* W_{r,s}^*| \\
&= |W_{r,s} W_{a,b} g W_{c,d}^* W_{r,s}^* \ket \bra W_{r,s} W_{a,b} g W_{c,d}^* W_{r,s}^*| \\
&= |\theta_{(r,s)}^* (W_{a,b} g W_{c,d}^*)\ket \bra \theta_{(r,s)}^*(W_{a,b} g W_{c,d}^*)| \\
&= \theta_{(r,s)}^* u^g_{(a,b),(c,d)} \theta_{(r,s)}
\end{align*} 
as required, using again Lemma \ref{weyllemma} $ b) $ and noting that $ \theta_{(r,s)}^*(T) = W_{r,s} T W_{r,s}^* $. 
\end{proof} 

Using equivariance of the Weyl quantum permutations, we can construct examples of quantum deck transformations for certain nontrivial coverings. For simplicity, we  only discuss the simplest case $ A = \mathbb{Z}/2\mathbb{Z} $ as follows. 

\begin{example} \label{equivariantdeckfinite}
Consider the torus $ X = S^1 \times S^1 $, and let $ E|X $ be the $ 4 $-sheeted connected covering space obtained by taking the cartesian product of two 
copies of the covering space from Example \ref{twosheeted}. 

Moreover let $ g \in U(2) $ such that $ \pi^g = (\mathbb{C}^4, u^g) $ is an irreducible Weyl quantum permutation of the fibre $ \mathbb{Z}/2\mathbb{Z} \times \mathbb{Z}/2 \mathbb{Z} $. Using the notation from Lemma \ref{weylequivariant}, and the identification $ \mathbb{C}^4 \cong M_2(\mathbb{C}) $, one checks that the unitaries $ \theta_{(a,b)} \in U(4) $ for $ a,b \in \mathbb{Z}/2 \mathbb{Z} $ have determinant one. 
Choose a continuous map $ \gamma $ from the boundary of the unit square in $ \mathbb{R}^2 $ to $ SU(4) $ such that 
\begin{align*}
\gamma(0,0) = \id, \quad \gamma(s,1) = \theta_{(0,1)} \gamma(s,0), \quad \gamma(1,t) = \theta_{(1,0)} \gamma(0,t) 
\end{align*}
for all $ s, t \in [0,1] $. This is possible since $ SU(4) $ is connected and $ \theta_{(1,0)} \theta_{(0,1)} = \theta_{(1,1)} = \theta_{(0,1)} \theta_{(1,0)} $. 
Since $ SU(4) $ is simply connected, this map can be extended to a continuous map $ [0,1] \times [0,1] \rightarrow SU(4) $, which we shall again denote by $ \gamma $. 

Now let us define projections $ v^g_{(x,a,b), (y,c,d)} \in M_4(\mathbb{C}) $ for $ x,y \in [0,1] \times [0,1] $ and $ (a,b), (c,d) \in \mathbb{Z}/2\mathbb{Z} \times \mathbb{Z}/2 \mathbb{Z} $ by 
$$
v^g_{(x,a,b), (y,c,d)} = \delta_{x,y} \gamma(x)^* u^g_{(a,b),(c,d)} \gamma(x).
$$
By construction, these projections assemble to a quantum homeomorphism $ (\mathbb{C}^4, v^g) $ of $ [0,1] \times [0,1] \times \mathbb{Z}/2 \mathbb{Z} \times \mathbb{Z}/2 \mathbb{Z} $. 
Moreover, using the properties of $ \gamma $ and Lemma \ref{weylequivariant}, we get 
\begin{align*}
v^g_{((s,1),a,b), ((s,1),c,d)} &= \gamma(s,1)^* u^g_{(a,b),(c,d)} \gamma(s,1) \\
&= \gamma(s,0)^* \theta_{(0,1)}^* u^g_{(a,b),(c,d)} \theta_{(0,1)} \gamma(s,0) \\
&= \gamma(s,0)^* u^g_{(a,b + 1),(c,d + 1)} \gamma(s,0) \\
&= v^g_{((s,0),a,b + 1), ((s,0),c,d + 1)}, 
\end{align*}
and similarly
\begin{align*}
v^g_{((1,t),a,b), ((1,t),c,d)} &= \gamma(1,t)^* u^g_{(a,b),(c,d)} \gamma(1,t) \\
&= \gamma(0,t)^* \theta_{(1,0)}^* u^g_{(a,b),(c,d)} \theta_{(1,0)} \gamma(0,t) \\
&= \gamma(0,t)^* u^g_{(a + 1,b),(c + 1,d)} \gamma(0,t) \\
&= v^g_{((0,t),a + 1,b), ((0,t),c + 1,d)} 
\end{align*}
for all $ s,t \in [0,1] $. 
Thus, if we view $ S^1 = [0,1]/\sim $ as the quotient of the unit interval given by identifying the endpoints, it follows that $ (\mathbb{C}^4, v^g) $ induces a quantum deck transformation of the covering space $ E|X $. Moreover, the induced quantum permutations of the fibres of $ E|X $ are all irreducible by construction. 
\end{example} 

If $ E|X $ is a covering space then we shall say that a quantum deck transformation $ \sigma $ is \emph{locally constant} if there exists an open covering $ (U_j)_{j \in J} $ of the base space $ X $ such that $ E_{|U_j} \cong U_j \times I $ is trivial and the induced quantum deck transformation $ \sigma_{|U_j} $ of $ E|_{U_j} $ is deck equivalent to a constant quantum deck transformation for every $ j \in J $. All classical deck transformations are locally constant in this sense. It is straightforward to check that the quantum deck transformations constructed in Examples \ref{extrivialquantumdeck}, \ref{twosheeted},  \ref{extwosphere} and \ref{equivariantdeckfinite} are locally constant as well. 

Locally constant quantum deck transformations are still ``tame'' in the sense that they are essentially determined by their restriction to a single fiber and the topology of the covering. 
To conclude, let us exhibit examples of quantum deck transformations which are not locally constant. For the sake of definiteness we consider the case of the 
covering of $ X = S^1 $ corresponding to the subgroup $ 4 \mathbb{Z} \subset \mathbb{Z} $. 

\begin{prop} \label{foursheetedcovering}
Let $ E = X = U(1) $ and let $ p: E \rightarrow X $ be given by $ p(z) = z^4 $. Then there 
exist $ 4 $-dimensional quantum deck transformations of $ E|X $ which are not locally constant. 
\end{prop} 

\begin{proof} 
Let $ \tau = (\mathbb{C}^4, v) $ be an irreducible $ 4 $-dimensional quantum permutation of $ \mathbb{Z}/4 \mathbb{Z} $. Then $ w_{i,j} = v_{i - 1, j - 1} $ defines another irreducible $ 4 $-dimensional quantum permutation $ \zeta = (\mathbb{C}^4, w) $ of $ \mathbb{Z}/4 \mathbb{Z} $. Note that if $ \sigma: \mathbb{Z}/4 \mathbb{Z} \rightarrow \mathbb{Z}/4 \mathbb{Z}, \sigma(i) = i + 1 $ denotes the shift 
permutation then $ \zeta $ can be written as $ \zeta = \sigma \tp \tau \tp \sigma^{-1} $. 

Let us fix a bijection between $ \mathbb{Z}/4\mathbb{Z} $ and $ \mathbb{Z}/2 \mathbb{Z} \times \mathbb{Z}/2\mathbb{Z} $. Then, by the structure of irreducible quantum permutations of four points, see Theorem \ref{BBfourpoints}, we find $ g, h \in U(2) $ such that $ \tau \cong \pi_g $ 
and $ \zeta \cong \pi_h $. 
Choosing a path $ \gamma: [0,1] \rightarrow U(2) $ connecting $ g $ to $ h $, and possibly conjugating the corresponding Weyl quantum permutations by a path of unitaries in $ U(4) $,
we obtain continuous maps $ u_{i,j}: [0,1] \rightarrow M_4(\mathbb{C}) $ such that $ (u_{i,j}(t)) $ is a magic unitary for all $ t $ 
and $ u_{i,j}(0) = v_{i,j}, u_{i,j}(1) = w_{i,j} $ for all $ i,j $. This allows us to define a quantum deck transformation $ (\mathbb{C}^4, u^E) $ of $ E $ by identifying $ X \cong [0,1]/\sim $ and  
$ E_{[0,1)} \cong [0,1) \times \mathbb{Z}/4 \mathbb{Z} $, and setting  
$$
u^E_{(x,i)(y,j)} = \delta_{x,y} u_{i,j}(x)
$$
for $ x,y \in [0, 1) $ and $ i,j \in \mathbb{Z}/4\mathbb{Z} $. 

The isomorphism class of the induced quantum permutations of the fibres does not remain locally constant over the base in this case. In particular, the quantum deck transformation $ (\mathbb{C}^4, u^E) $ is not locally constant.  
\end{proof} 

Needless to say, for $ n $-sheeted coverings with $ n > 4 $, the situation gets more complicated than in our examples above, and a complete description of all quantum deck transformations is not feasible.

\subsection{Exercises}

\begin{exercise}
What is the universal property of the quantum homeomorphism group of a locally compact space? 
\end{exercise}

\begin{exercise} \label{ex142}
Show that if we require strict equality $ u_{x \cdot g, y \cdot g} = u_{x,y} $ in Definition \ref{gequiv} and $ G $ acts transitively on $ X $ then $ (\Hc, u) $ must be classical. 
\end{exercise}

\begin{exercise}
Generalise the constructions from Example \ref{equivariantdeckfinite} to the covering of $ X = S^1 \times S^1 $ given by the cartesian product of two copies of the connected covering of $ S^1 $ with $ N $ sheets for $ N > 2 $. 
\end{exercise}

\end{document}